\documentclass[12pt,twoside]{amsart}
\usepackage{geometry}
\usepackage{amssymb,amsmath,amsthm, amscd, enumerate, mathrsfs}
\usepackage{graphicx, hhline}
\usepackage[all]{xy}
\usepackage[shortlabels]{enumitem}
\setlist[enumerate]{leftmargin=56pt,labelsep=
8pt,itemsep=4pt,label=\upshape{(\thethm.\arabic*)}}
\usepackage{mathtools}
\mathtoolsset{showonlyrefs}

\usepackage{iftex}

\ifPDFTeX
  \RequirePackage{hyperref}
\else
  \RequirePackage[dvipdfmx]{hyperref}
\fi

\usepackage{xcolor}

\title{On finiteness of relative log 
pluricanonical representations}
\author{Osamu Fujino}
\date{2026/7/29, version 0.25}
\address{Department of 
Mathematics, Graduate School of Science, 
Kyoto University, Kyoto 606-8502, Japan}
\email{fujino@math.kyoto-u.ac.jp}

\DeclareMathOperator{\Nlc}{Nlc}
\DeclareMathOperator{\Exc}{Exc}
\DeclareMathOperator{\A}{A}
\DeclareMathOperator{\PA}{PA}
\DeclareMathOperator{\NE}{NE}
\DeclareMathOperator{\Supp}{Supp}
\DeclareMathOperator{\Nklt}{Nklt}

\DeclareMathOperator{\lcm}{lcm}
\DeclareMathOperator{\id}{id}
\DeclareMathOperator{\Bir}{Bir}
\DeclareMathOperator{\xIm}{Im}
\DeclareMathOperator{\ev}{ev}
\DeclareMathOperator{\GL}{GL}
\DeclareMathOperator{\Aut}{Aut}
\DeclareMathOperator{\Bim}{Bim}

\newtheorem{thm}{Theorem}[section]
\newtheorem{lem}[thm]{Lemma}
\newtheorem{cor}[thm]{Corollary}
\newtheorem{prop}[thm]{Proposition}
\newtheorem{conj}[thm]{Conjecture}

\theoremstyle{definition}
\newtheorem{defn}[thm]{Definition}
\newtheorem{rem}[thm]{Remark}
\newtheorem{ex}[thm]{Example}
\newtheorem*{ack}{Acknowledgments}  
\newtheorem{step}{Step}
\newtheorem*{case}{Case}
\newtheorem{say}[thm]{}
\newtheorem*{claim}{Claim}
\makeatletter

\@addtoreset{equation}{section}
\makeatother
\begin{document}

\begin{abstract} 
We prove the finiteness of 
relative log pluricanonical representations in the 
complex analytic setting. 
As an application, we discuss the abundance conjecture for 
semi-log canonical pairs within this framework. 
Furthermore, we establish the existence of log canonical flips 
for complex analytic spaces.
Roughly speaking, 
we reduce the abundance conjecture for semi-log canonical 
pairs to the case of log canonical pairs in the complex analytic setting.
Moreover, 
we show that the abundance conjecture for 
projective morphisms of complex analytic spaces 
can be reduced to the classical abundance conjecture for projective varieties.
\end{abstract} 

\maketitle

\tableofcontents 

\section{Introduction}\label{b-sec1} 

In the very early stages of Mori's minimal model theory, Noboru Nakayama laid the foundations of the minimal model program for projective morphisms between complex analytic spaces in \cite{nakayama-lower} (see also \cite{nakayama}). Since then, to the best of our knowledge, a general theory of the minimal model program for projective morphisms between complex analytic spaces was not systematically developed for a long time. This framework was first revisited in \cite{fujino-bchm}, where the main results of \cite{bchm} were extended to the setting of projective morphisms between complex analytic spaces. Subsequently, through a somewhat different approach, \cite{dhp} obtained similar results, and \cite{lm} further broadened the applicability of the minimal model program to a wider class of spaces.

Since \cite{fujino-bchm}, in a series of papers \cite{fujino-vanishing}, \cite{fujino-cone-contraction}, \cite{fujino-quasi-log}, \cite{fujino-vanishing-pja}, \cite{fujino-fujisawa}, \cite{enokizono-hashizume-mmp}, among others, many of the results of the minimal model program known in the algebraic setting have been generalized to projective morphisms between complex analytic spaces. The present work can be viewed as a complex analytic generalization of \cite{fujino-gongyo} (see also \cite{fujino-abundance}). Based on the results of this paper, \cite{enokizono-hashizume-termination} and \cite{hashizume5} were subsequently written. Starting with \cite{fujino-bchm}, a series of papers by the author, Enokizono, and Hashizume, including the present paper, has established most of the general theory of the minimal model program in the complex analytic context. Furthermore, it has been shown that almost all open problems in the minimal model program in our complex analytic setting reduce to the original open problems in the algebraic setting. In any case, this paper forms an integral part of this effort to develop the minimal model program for projective morphisms between complex analytic spaces.

One of the main objectives of the present paper is to establish the following result related to the abundance conjecture:

\begin{thm}[{Abundance theorem for semi-log canonical pairs in the 
complex analytic setting, 
cf.~\cite[Theorem 1.5]{fujino-gongyo}}]\label{b-thm1.1}
Let $\pi\colon X\to Y$ be a projective morphism 
of complex analytic spaces, let $W$ be a 
compact subset of $Y$, and let $(X, \Delta)$ be a semi-log 
canonical pair such that $K_X+\Delta$ is $\mathbb Q$-Cartier. 
Let $\nu\colon X^\nu\to X$ be the normalization. 
Assume that $K_{X^\nu}+\Theta:=\nu^*(K_X+\Delta)$ 
is $\pi\circ \nu$-semiample over some open neighborhood of 
$W$. 
Then there exists an open neighborhood $U$ of $W$ and 
a divisible positive integer $m$ 
such that $\mathcal O_X(m(K_X+\Delta))$ is $\pi$-generated 
over $U$.  
\end{thm}

In order to prove Theorem \ref{b-thm1.1}, we need: 

\begin{thm}[{Finiteness of relative log pluricanonical 
representations, I, cf.~\cite[Theorem 1.1]{fujino-gongyo}}]\label{b-thm1.2}
Let $\pi\colon X\to Y$ be a projective morphism from a {\em{(}}not 
necessarily connected{\em{)}} normal 
complex analytic space $X$ onto a complex variety $Y$ such that 
$(X, \Delta)$ is log canonical and that every irreducible 
component of $X$ is dominant onto $Y$. 
Let $m$ be a positive integer such that 
$m(K_X+\Delta)$ is Cartier and $\pi_*\mathcal O_X(m(K_X+\Delta))
\ne 0$. 
Assume that $K_X+\Delta$ is $\pi$-semiample. 
Then the image of 
\begin{equation}
\rho_m\colon \Bim (X/Y, \Delta)\to \Aut_{\mathcal O_Y}\left(
\pi_*\mathcal O_X(m(K_X+\Delta))\right)
\end{equation} 
is a finite group, where $\Bim(X/Y, \Delta)$ is 
the group of all $B$-bimeromorphic maps of $(X, \Delta)$ over $Y$. 
\end{thm}

As an easy consequence of Theorem \ref{b-thm1.2}, 
we have a useful corollary. 
We will use it in the proof of Theorem \ref{b-thm1.1}. 

\begin{cor}[{Finiteness of relative log pluricanonical 
representations, II, cf.~\cite[Theorem 1.1]{fujino-gongyo}}]\label{b-cor1.3}
Let $(X, \Delta)$ be an equidimensional {\em{(}}not necessarily 
connected{\em{)}} log canonical pair and let $\pi\colon X\to Y$ be a projective 
morphism of complex analytic spaces. 
Let $m$ be a positive integer such that 
$m(K_X+\Delta)$ is Cartier and $\pi_*\mathcal O_X(m(K_X+\Delta))
\ne 0$. 
Assume that $K_X+\Delta$ is $\pi$-semiample. Let 
$W$ be a compact subset of $Y$ and let 
$U$ be a semianalytic Stein open subset of $Y$ with $U\subset W$. 
Let $\Bim (X/Y, \Delta; W)$ be the group of 
all $B$-bimeromorphic maps $g$ defined over some open 
neighborhood $U_g$ of $W$. In this setting, 
we can consider 
\begin{equation}\label{b-eq1.1}
\rho^{WU}_m\colon \Bim(X/Y, \Delta; W)\to 
\Aut_{\mathcal O_U}\left(\pi_*\mathcal O_{\pi^{-1}(U)}(m(K_X+\Delta))\right). 
\end{equation} 
Then $\rho^{WU}_m(\Bim(X/Y, \Delta; W))$ is a finite group. 
\end{cor}

As an application of Theorem \ref{b-thm1.1}, 
we have: 

\begin{thm}[{Freeness for nef and log abundant log canonical 
bundles, cf.~\cite[Theorem 1.6]{fujino-gongyo}}]\label{b-thm1.4}
Let $(X, \Delta)$ be a semi-log canonical 
pair and let $\pi\colon X\to Y$ be a projective 
morphism of complex analytic spaces. 
Assume that $K_X+\Delta$ is $\mathbb Q$-Cartier 
and is $\pi$-nef and $\pi$-log abundant with respect to 
$(X, \Delta)$ over $Y$. 
Let $W$ be a compact subset of $Y$. 
Then there exists a positive integer $m$ such that 
$\mathcal O_X(m(K_X+\Delta))$ is $\pi$-generated 
over some open neighborhood of $W$. 
\end{thm}

Theorem \ref{b-thm1.4} is well known when
$\pi\colon X \to Y$ is algebraic (see \cite[Theorem 1.6]{fujino-gongyo}).
As mentioned above, we prove Theorem \ref{b-thm1.4}
as a consequence of Theorem \ref{b-thm1.1}.
In our proof of Theorem \ref{b-thm1.4} presented in the present paper,
we make use of a kind of canonical bundle formula
(see \cite{fujino-kawamata} and \cite{fujino-basepoint-free}).
Therefore, the result is not entirely obvious.
When $K_X + \Delta$ is only assumed to be $\mathbb{R}$-Cartier,
we have the following theorem.

\begin{thm}\label{b-thm1.5}
Let $\pi\colon X\to Y$ be a projective 
morphism of complex analytic spaces and 
let $W$ be a Stein compact subset 
of $Y$ such that $\Gamma (W, \mathcal O_Y)$ is noetherian. 
Let $U$ be an open subset of $Y$ and 
let $L$ be a compact subset of $Y$ such that $L\subset U\subset W$. 
Let $(X, \Delta)$ be a log canonical 
pair such that $K_X+\Delta$ is $\pi$-nef and 
$\pi$-log abundant with respect to $(X, \Delta)$ over $Y$. 
Then $K_X+\Delta$ is $\pi$-semiample over 
some open neighborhood of $L$. 
\end{thm}

We note that Stein compact subsets play an important 
role in \cite{fujino-bchm}. 

\begin{rem}[Stein compact subsets]\label{b-rem1.6}
A compact subset of an analytic space is said to 
be {\em{Stein compact}} if it admits a fundamental 
system of Stein open neighborhoods. 
Let $W$ be a Stein compact subset on a complex 
analytic space $Y$. 
Then, by Siu's theorem, 
$\Gamma(W, \mathcal O_Y)$ is noetherian if and only 
if $W\cap Z$ has only finitely many connected 
components for any analytic subset $Z$ which 
is defined over an open neighborhood 
of $W$. 
Hence, if $W$ is a Stein compact semianalytic subset of a 
complex analytic space $Y$, then 
$\Gamma (W, \mathcal O_Y)$ is always noetherian. 
\end{rem}

By combining Theorem \ref{b-thm1.5} with 
\cite[Theorem 1.2]{enokizono-hashizume-mmp}, 
we can prove the existence of log canonical flips in the 
complex analytic setting. 
We learned it from Kenta Hashizume. 

\begin{thm}[Existence of log canonical flips]\label{b-thm1.7}
Let $\varphi\colon X\to Z$ be a small projective 
bimeromorphic morphism of normal complex varieties such that 
$(X, \Delta)$ is log canonical and that $-(K_X+\Delta)$ is 
$\varphi$-ample. Then we have a commutative diagram 
\begin{equation*}
\xymatrix{
(X, \Delta) \ar[dr]_-\varphi
\ar@{-->}[rr]^-\phi&& (X^+, \Delta^+)\ar[dl]^-{\varphi^+}\\ 
& Z & 
}
\end{equation*}
satisfying the following properties: 
\begin{itemize}
\item[{\em{(i)}}] $\varphi^+\colon X^+\to Z$ is a small 
projective bimeromorphic morphism of normal complex 
varieties, 
\item[{\em{(ii)}}] $(X^+, \Delta^+)$ is log canonical, 
where $\Delta^+$ is the strict transform of $\Delta$ on $X^+$, 
and 
\item[{\em{(iii)}}] $K_{X^+}+\Delta^+$ is $\varphi^+$-ample. 
\end{itemize}
We usually simply say that $\phi\colon (X, \Delta)\dashrightarrow 
(X^+, \Delta^+)$ is a log canonical flip. 
\end{thm}

By combining Theorem \ref{b-thm1.5} with \cite[Theorem 1.3]{enokizono-hashizume-mmp},
we establish the existence of good dlt blow-ups 
in the complex analytic setting.
This result immediately yields the inversion 
of adjunction for log canonicity (see \cite[Theorem 1.1]{fujino-inversion} 
and Theorem \ref{y-thm6.3}).
 
\begin{thm}[Good dlt blow-ups]\label{b-thm1.8}
Let $X$ be a normal complex variety and let $\Delta$ be an effective 
$\mathbb R$-divisor on $X$ such that 
$K_X+\Delta$ is $\mathbb R$-Cartier. 
Note that $\Delta$ is not necessarily a boundary $\mathbb R$-divisor. 
Let $W$ be a Stein compact 
subset of $X$ such that $\Gamma (W, \mathcal O_X)$ is noetherian. 
Then, after shrinking $X$ around $W$ suitably, 
we can construct a projective bimeromorphic morphism 
$f\colon Z\to X$ from a normal complex variety $Z$ with the following 
properties: 
\begin{itemize}
\item[{\em{(i)}}] $Z$ is $\mathbb Q$-factorial over $W$, 
\item[{\em{(ii)}}] $a(E, X, \Delta)\leq -1$ for every $f$-exceptional 
divisor $E$ on $Z$, and 
\item[{\em{(iii)}}] $(Z, \Delta^{<1}_Z+\Supp \Delta^{\geq 1}_Z)$ is divisorial 
log terminal, where $K_Z+\Delta_Z=f^*(K_X+\Delta)$. 
\end{itemize}
Moreover, we put 
\[
\Delta^\dag_Z:=\Delta^{<1}_Z+\Supp \Delta^{\geq 1}_Z=
\Delta^{\leq 1}_Z+\Supp \Delta^{>1}_Z. 
\] 
Then we have 
\[
K_Z+\Delta^\dag_Z=f^*(K_X+\Delta)-G, 
\] 
where 
\[
G=\Delta^{\geq 1}_Z-\Supp \Delta^{\geq 1}_Z=\Delta^{>1}_Z-\Supp \Delta^{>1}_Z. 
\] 
In this setting, we can make $f\colon Z\to X$ satisfy: 
\begin{itemize}
\item[{\em{(iv)}}] $-G$ is $f$-nef over $W$. 
\end{itemize}
We further assume that 
\[
L\subset U'\subset W' \subset U\subset W, 
\] 
where $W'$ is a Stein compact subset of $X$ such that 
$\Gamma (W', \mathcal O_X)$ is noetherian, 
$U$ and $U'$ are open subsets of $X$, and $L$ is a compact subset of $X$. 
Then, by Theorem \ref{b-thm1.5}, we have: 
\begin{itemize}
\item[{\em{(v)}}] $-G$ is $f$-semiample over some open neighborhood 
of $L$. 
\end{itemize}
\end{thm}

In Section \ref{y-sec6}, we 
also present another refinement 
of dlt blow-ups (see 
Theorem \ref{y-thm6.1}), 
derived as a consequence of 
\cite{enokizono-hashizume-mmp}, which is closely related to Theorem \ref{b-thm1.8}. 
Both Theorems \ref{b-thm1.8} and \ref{y-thm6.1} are useful 
for the study of complex analytic singularities. 
As applications of these theorems, we discuss 
the ACC for log canonical thresholds (see Theorem \ref{y-thm6.2}) 
and the inversion of adjunction for log canonicity 
(see Theorem \ref{y-thm6.3}) in the complex analytic setting. 

For the reader's convenience, 
we recall the abundance conjecture for projective log canonical pairs.
It is well known that the abundance conjecture 
is among the most important and profound conjectures in the theory of minimal models.

\begin{conj}[Abundance conjecture for projective log canonical 
pairs]\label{b-conj1.9}
Let $(X, \Delta)$ be a projective log canonical pair such that 
$K_X+\Delta$ is nef. 
Then $K_X+\Delta$ is semiample. 
\end{conj}

It is well known that Conjecture \ref{b-conj1.9} 
has already been solved in $\dim X\leq 3$. 
When $\dim X\geq 4$, it is still widely open. 
By Theorem \ref{b-thm1.4}, we have: 

\begin{thm}[{cf.~\cite[Theorem 1.30]{fujino-bchm}}]\label{b-thm1.10} 
Assume that Conjecture \ref{b-conj1.9} holds in dimension $n$. 
Let $\pi\colon X\to Y$ be a projective surjective morphism 
of normal complex varieties with $\dim X\leq n$ and 
let $(X, \Delta)$ be a log canonical pair such that $K_X+\Delta$ is 
$\mathbb Q$-Cartier. 
Assume that $K_X+\Delta$ is $\pi$-nef. 
Let $W$ be a compact subset of $Y$. 
Then there exists a positive integer $m$ such that 
$\mathcal O_X(m(K_X+\Delta))$ is $\pi$-generated 
over some open neighborhood of $W$. 
\end{thm}

When $K_X+\Delta$ is only $\mathbb R$-Cartier, 
we have: 

\begin{cor}\label{b-cor1.11}
Assume that Conjecture \ref{b-conj1.9} holds in dimension $n$. 
Let $\pi\colon X\to Y$ be a projective surjective morphism 
of normal complex varieties with $\dim X\leq n$ and 
let $(X, \Delta)$ be a log canonical pair. Assume that 
$K_X+\Delta$ is $\pi$-nef. 
Let $W$ be a Stein compact subset of $Y$ such that $\Gamma (W, \mathcal O_Y)$ 
is noetherian. Let $U$ be an open subset of $Y$ and let $L$ be a 
compact subset of $Y$ with $L\subset U\subset W$. 
Then $K_X+\Delta$ is $\pi$-semiample over 
some open neighborhood of $L$. 
\end{cor}

Theorem \ref{b-thm1.10} and Corollary \ref{b-cor1.11} 
show that the abundance conjecture for projective 
morphisms of complex analytic spaces can be reduced 
to the original abundance conjecture for projective 
varieties. Therefore, in order to address the abundance 
conjecture for projective morphisms between complex 
analytic spaces, it suffices to resolve 
Conjecture \ref{b-conj1.9}. In the case where 
$(X, \Delta)$ is a kawamata log terminal pair, 
Theorem \ref{b-thm1.10} has already been established 
in \cite[Theorem 1.30]{fujino-bchm}.

Note that in the present paper, we employ the minimal 
model program for projective morphisms between complex 
analytic spaces, as established in \cite{fujino-bchm}, 
\cite{enokizono-hashizume-ss}, and \cite{enokizono-hashizume-mmp}. 
Additionally, we make use of the vanishing theorems proved 
in \cite{fujino-vanishing} (see also \cite{fujino-vanishing-pja} 
and \cite{fujino-fujisawa}). However, we do not 
employ Koll\'ar's gluing theory as presented 
in \cite{kollar}, since it is currently unclear whether 
it applies to complex analytic spaces.

\begin{rem}[see Lemma \ref{p-lem4.11}]\label{b-rem1.12} 
We can easily check that Theorem \ref{b-thm1.1} recovers 
\cite[Theorem 2]{hacon-xu}, 
which is the original algebraic version of this problem. 
Hence the present paper gives an alternative proof 
of \cite[Theorem 2]{hacon-xu} 
without using Koll\'ar's gluing 
theory in \cite{kollar}. 
\end{rem}

\begin{rem}\label{b-rem1.13} 
As mentioned above, a series of papers starting with \cite{fujino-bchm} systematically develops the minimal model theory for projective morphisms between complex analytic spaces. Some caution is needed, as the order of writing these papers does not coincide with the order of their publication. Particular attention should be paid to \cite{enokizono-hashizume-termination} and \cite{hashizume5}, where the latter is a complex analytic generalization of \cite{hashizume4}. Since \cite{enokizono-hashizume-termination} and \cite{hashizume5} rely on the results established in the present paper, such as Theorems \ref{b-thm1.5} and \ref{b-thm1.7}, we do not cite or use any results from these two papers here to avoid circular arguments. On the other hand, papers such as \cite{fujino-bchm}, \cite{fujino-vanishing}, \cite{fujino-cone-contraction}, \cite{fujino-quasi-log}, \cite{enokizono-hashizume-ss}, \cite{enokizono-hashizume-mmp} were written independently of the present paper, and thus their results will be used freely throughout this paper.

Based on the aforementioned papers and the present paper, we believe that various results of the minimal model program for log canonical pairs can be formulated and proved in the complex analytic setting. We do not discuss these details here; for further developments and details, see the papers mentioned above and the references therein.
\end{rem}

\begin{rem}\label{b-rem1.14} 
In \cite{fujino-surfaces}, 
we demonstrated that the minimal model 
theory for algebraic surfaces can be developed 
under significantly weaker assumptions than those 
required in higher dimensions. A comparable result 
holds for projective morphisms between complex analytic 
spaces. Moriyama (see \cite{moriyama}) offers a detailed 
treatment of the minimal model theory for surfaces 
in the complex analytic setting.
\end{rem}

We now outline the organization of the present paper. 
In Section \ref{c-sec2}, we review basic definitions 
and results that are essential for the development of the present paper. 
Section \ref{a-sec3} is devoted to the finiteness of 
relative log pluricanonical representations. 
Our proof of Theorem \ref{b-thm1.2} relies on the 
finiteness of log pluricanonical representations 
for projective log canonical pairs, as established in \cite{fujino-gongyo}. 
In Section \ref{p-sec4}, we address the abundance 
conjecture for semi-log canonical pairs in the complex analytic 
setting. More specifically, we prove Theorem \ref{b-thm1.1}, which 
is one of the main results of the present paper. 
In Section \ref{p-sec5}, we prove Theorem \ref{b-thm1.4} as an 
application of Theorem \ref{b-thm1.1}, and then deduce 
Theorem \ref{b-thm1.10} as a straightforward consequence of 
Theorem \ref{b-thm1.4}. We also 
establish Theorems \ref{b-thm1.5}, \ref{b-thm1.7}, and Corollary \ref{b-cor1.11}. 
In Section \ref{y-sec6}, 
we discuss dlt blow-ups as applications of the minimal model program 
for log canonical pairs, as established in \cite{enokizono-hashizume-mmp}. 
We prove two generalizations of dlt blow-ups (see Theorems \ref{b-thm1.8} and 
\ref{y-thm6.1}). As a direct application of Theorem \ref{y-thm6.1}, 
we address the ACC for log canonical thresholds in the setting of 
complex analytic spaces (see Theorem \ref{y-thm6.2}). 
Moreover, we show that Theorem \ref{b-thm1.8} 
allows us to quickly recover the inversion of adjunction 
for log canonicity (see Theorem \ref{y-thm6.3}). 
Section \ref{f-sec7} is an appendix, where we discuss several properties 
of extremal Fano contractions. Although these results are well known in 
the original algebraic setting, we record them here for completeness, 
as no suitable reference seems to be available in the complex analytic setting. In the final section, Section \ref{x-sec8}, we provide some supplementary 
comments on \cite{fujino-abundance} and \cite{fujino-gongyo} for 
the reader's convenience. We also correct some minor issues in these papers.

\begin{ack}\label{b-ack}
The author was partially supported by JSPS KAKENHI Grant Numbers 
JP19H01787, JP20H00111, JP21H04994, and JP23K20787. 
The author would like to thank Professor Shigefumi Mori 
for his warm encouragement. He is also deeply grateful to 
Makoto Enokizono, Taro Fujisawa, Yoshinori Gongyo, Kenta Hashizume, 
and Nao Moriyama 
for valuable discussions and support. 
Finally, he would like to thank the referee for helpful comments and suggestions.
\end{ack}

In the present paper, every complex 
analytic space is assumed to be {\em{Hausdorff}} 
and {\em{second-countable}}. 
A reduced and irreducible complex analytic 
space is called a {\em{complex variety}}. 
We will freely use the basic results on 
complex analytic geometry in \cite{bs} and \cite{fi}. 
For the minimal model program for projective 
morphisms between complex analytic spaces, 
see \cite{fujino-bchm} (see also \cite{enokizono-hashizume-ss} and 
\cite{enokizono-hashizume-mmp}). 
For the basic definitions and results in the 
theory of minimal models for 
higher-dimensional algebraic varieties, 
see \cite{fujino-fundamental} and 
\cite{fujino-foundations} (see also \cite{kollar-mori} and 
\cite{kollar}). In the present paper, we sometimes 
use semianalytic sets. For the basic properties of semianalytic sets, 
see \cite{bierstone-milman-semi}. 

\section{Preliminaries}\label{c-sec2}

In this section, we collect some basic definitions and results
necessary for the present paper.
We begin with the following fundamental definitions.

\begin{defn}[{\cite[Definition 2.32]{fujino-bchm} and 
\cite[2.1.6]{fujino-cone-contraction}}]\label{c-def2.1}
Let $X$ be a normal complex variety and let $D=\sum _i a_i D_i$ 
be an $\mathbb R$-divisor on $X$ such that 
$D_i$ is a prime divisor 
on $X$ for every $i$ with $D_i\ne D_j$ for $i\ne j$. 
We put 
\begin{equation}
\lfloor D\rfloor :=\sum _i \lfloor a_i \rfloor D_i, \quad 
\lceil D\rceil :=-\lfloor -D\rfloor, \quad \text{and} 
\quad \{D\}:=D-\lfloor D\rfloor. 
\end{equation}
We also put 
\begin{equation}
D^{=1}:=\sum _{a_i=1} D_i, \quad 
D^{<1}:=\sum _{a_i<1} a_i D_i, 
\quad \text{and} \quad 
D^{>1}:=\sum _{a_i>1} a_i D_i. 
\end{equation} 
Similarly, we can define $D^{\leq 1}$ and $D^{\geq 1}$. 
We note that $D$ is called a {\em{boundary $\mathbb Q$-divisor}} 
(resp.~a {\em{subboundary $\mathbb Q$-divisor}}) 
when $a_i\in \mathbb Q$ and $0\leq a_i\leq 1$ (resp.~$a_i\leq 1$) for every $i$. 
\end{defn}

Let us recall the definitions of log canonical pairs and 
log canonical strata. For a detailed discussion 
of the singularities of pairs, see \cite{fujino-fundamental}, 
\cite{fujino-foundations}, \cite[Section 3]{fujino-bchm}, 
\cite[Section 2.1]{fujino-cone-contraction}, 
\cite{kollar}, and others. Although there are some subtle 
issues regarding the complex analytic singularities of pairs, we do 
not repeat the details here.

\begin{defn}[{Log canonical pairs and log canonical 
strata, see \cite[Definition 3.1]{fujino-bchm} 
and \cite[2.1.1]{fujino-cone-contraction}}]\label{c-def2.2} 
Let $X$ be a normal complex analytic space and 
let $\Delta$ be an effective 
$\mathbb R$-divisor on $X$ such that 
$K_X+\Delta$ is $\mathbb R$-Cartier. 
If $a(E, X, \Delta)\geq -1$ (resp.~$>-1$) holds 
for any proper bimeromorphic 
morphism $f\colon Y\to X$ from a normal 
complex analytic space $Y$ and every $f$-exceptional 
divisor $E$, 
then $(X, \Delta)$ is called a {\em{log canonical}} 
(resp.~{\em{purely log terminal}}) 
{\em{pair}}. If $(X, \Delta)$ is purely log terminal 
and $\lfloor \Delta\rfloor=0$, then we say that 
$(X, \Delta)$ is a 
{\em{kawamata log terminal pair}}. 

Let $(X, \Delta)$ be a log canonical pair. 
The image of $E$ with $a(E, X, \Delta)=-1$ for 
some $f\colon Y\to X$ is called a {\em{log canonical 
center}} of $(X, \Delta)$. 
A closed subset $S$ of $X$ is called a {\em{log 
canonical stratum}} of $(X, \Delta)$ if $S$ is an irreducible 
component of $X$ or 
a log canonical center of $(X, \Delta)$. 
\end{defn}
 
\begin{defn}[Non-lc loci and non-klt loci]\label{c-def2.3} 
Let $X$ be a normal complex variety and let $\Delta$ be an effective 
$\mathbb R$-divisor on $X$ such that $K_X+\Delta$ is $\mathbb R$-Cartier. 
Then the {\em{non-lc locus}} (resp.~{\em{non-klt locus}}) of $(X, \Delta)$, 
denoted by $\Nlc(X, \Delta)$ (resp.~$\Nklt(X, \Delta)$), is the smallest 
closed subset $Z$ of $X$ such that 
the complement $(X\setminus Z, \Delta|_{X\setminus Z})$ is log 
canonical (resp.~kawamata log terminal). 

Although the non-klt locus $\Nklt(X, \Delta)$ has a natural complex analytic 
space structure defined by the multiplier ideal sheaf $\mathcal J(X, \Delta)$, 
this structure does not play a major role in this paper.
\end{defn}
 
Let us recall the definition of divisorial log terminal pairs 
in the complex analytic setting (see \cite[Definition 3.7]{fujino-bchm}). 
Note that \cite[Definition 2.37, Proposition 2.40, Theorem 2.44]{kollar-mori} 
is helpful. 

\begin{defn}[Divisorial log terminal pairs]\label{c-def2.4} 
Let $X$ be a normal complex analytic space and 
let $\Delta$ be a boundary $\mathbb R$-divisor 
on $X$ such that $K_X+\Delta$ is $\mathbb R$-Cartier. 
If there exists a proper bimeromorphic 
morphism $f\colon Y\to X$ from a smooth 
complex variety $Y$ such that $\Exc(f)$ and $\Exc(f)\cup 
\Supp f^{-1}_*\Delta$ are simple normal crossing divisors on $Y$ and 
that the discrepancy coefficient $a(E, X, \Delta)>-1$ holds 
for every $f$-exceptional 
divisor $E$, then $(X, \Delta)$ is called a 
{\em{divisorial log terminal pair}}. 
We note that $\Exc(f)$ denotes the {\em{exceptional locus}} 
of $f$. 
\end{defn}

We note that Definitions \ref{c-def2.2} and 
\ref{c-def2.4} work for 
a finite disjoint union of normal complex varieties. 
In Definitions \ref{c-def2.2} and 
\ref{c-def2.4}, $X$ is not necessarily connected. 
It is well known that a divisorial log terminal pair is 
a log canonical pair. 

\begin{rem}\label{c-rem2.5}
If we shrink $X$ to a relatively compact open subset of $X$ in Definition 
\ref{c-def2.4}, 
then we can assume that $f$ is a composite of a finite sequence of blow-ups. 
In particular, $f$ is projective. 
For the details, see \cite[Lemma 3.9]{fujino-bchm} and 
\cite{bierstone-milman}. 
\end{rem}

Let us define semi-log canonical 
pairs and 
semi-divisorial 
log terminal pairs 
in the complex analytic setting. 

\begin{defn}[Semi-log canonical pairs and 
semi-divisorial log terminal pairs]\label{c-def2.6}
Let $X$ be an equidimensional reduced complex analytic space 
that is normal crossing in codimension one and satisfies 
Serre's $S_2$ condition. 
Let $\Delta$ be an effective $\mathbb R$-divisor 
on $X$ such that the singular locus of $X$ 
does not contain any irreducible components of $\Supp \Delta$. 
In this situation, the pair $(X, \Delta)$ is called 
a {\em{semi-log canonical pair}} 
(an {\em{slc pair}}, for short) if 
\begin{itemize}
\item[(1)] $K_X+\Delta$ is $\mathbb R$-Cartier, and 
\item[(2)] $(X^\nu, \Theta)$ is log canonical, where 
$\nu\colon X^\nu \to X$ is the normalization and 
$K_{X^\nu}+\Theta:=\nu^*(K_X+\Delta)$. 
\end{itemize} 
Let $(X, \Delta)$ be a semi-log canonical pair in the above sense. 
If each irreducible component of $X$ is normal and $(X^\nu, 
\Theta)$ is divisorial log terminal, 
then we say that $(X, \Delta)$ is a 
{\em{semi-divisorial log terminal pair}} 
(an {\em{sdlt pair}}, for short). 
Let $S$ be a closed subset of $X$. 
We say that $S$ is a {\em{semi-log canonical 
stratum}} of $(X, \Delta)$ if and only if 
$S$ is an irreducible component of $X$ or 
the $\nu$-image of some log canonical center of $(X^\nu, \Theta)$. 
When $(X, \Delta)$ is log canonical, then a semi-log 
canonical stratum $S$ is called a {\em{log canonical stratum}} of 
$(X, \Delta)$ (see Definition \ref{c-def2.2}).  
\end{defn}

For various results on algebraic (resp.~complex 
analytic) semi-log canonical pairs, 
see \cite{fujino-slc} (resp.~\cite{fujino-quasi-log}). 

\begin{rem}\label{c-rem2.7}
Note that the definition of semi-divisorial log terminal pairs 
in Definition \ref{c-def2.6} is different from \cite[Definition 5.19]{kollar}. 
Our definition is a direct analytic generalization of the one in 
\cite{fujino-abundance} (see \cite[Definition 1.1]{fujino-abundance}). 
\end{rem}

The following lemma is well known when $X$ is an algebraic variety. 
We state it here explicitly for the sake of completeness. 

\begin{lem}\label{c-lem2.8}
Let $(X, \Delta)$ be a divisorial log terminal pair. 
We put $S:=\lfloor \Delta\rfloor$ and $K_S+\Delta_S:=(K_X+\Delta)|_S$ 
by adjunction. 
Then $(S, \Delta_S)$ is semi-divisorial log terminal in the sense of 
Definition \ref{c-def2.6}. 
More precisely, let $S=S_1+\cdots +S_l$ be the irreducible decomposition. 
We put $T:=S_1+\cdots +S_l$ for 
some $l$ with $1\leq l\leq k$. 
Then $T$ is Cohen--Macaulay and is simple normal crossing in codimension one. 
In particular, every irreducible component of $S$ is normal. 
We put $K_{S_i}+\Delta_{S_i}:=(K_X+\Delta)|_{S_i}$ by adjunction for 
every $i$. 
Then $(S_i, \Delta_{S_i})$ is divisorial log terminal. 
Thus we see that $(T, \Delta_T)$, 
where $K_T+\Delta_T:=(K_X+\Delta)|_T$ by adjunction, 
is semi-divisorial log terminal. 
\end{lem}

\begin{proof}
By \cite{rrv}, we can apply the proof of 
\cite[Theorem 3.13.6]{fujino-foundations} to 
our setting with some suitable 
modifications (see also Remark \ref{c-rem2.5}). 
Then we obtain that $T$ is Cohen--Macaulay. 
It is obvious that $T$ is simple normal crossing in codimension 
one. 
Hence we can easily check all the other statements. 
\end{proof}

We will repeatedly use Lemma \ref{c-lem2.9} in subsequent 
sections. 

\begin{lem}\label{c-lem2.9}
Let $(X, \Delta)$ be a log canonical pair such that 
$(X, \Delta-\lfloor \Delta\rfloor)$ is kawamata log 
terminal. We put $S:=\lfloor \Delta\rfloor$ and 
$K_S+\Delta_S:=(K_X+\Delta)|_S$ by adjunction. 
Then $S$ is Cohen--Macaulay and is semi-log canonical. 
\end{lem}

\begin{proof}
Since $(X, \Delta-S)$ is kawamata log terminal, 
$X$ has only rational singularities. Therefore, 
$X$ is Cohen--Macaulay. 
Since $S$ is $\mathbb Q$-Cartier, $\mathcal O_X(-S)$ is 
Cohen--Macaulay. 
This implies that $\mathcal O_S$ is Cohen--Macaulay. 
For the details, see \cite[Corollary 5.25]{kollar-mori}, 
\cite[Corollaries 2.62, 2.63, and 2.88]{kollar}, and others. 
By adjunction, we see that 
$(S, \Delta_S)$ is semi-log canonical. 
\end{proof}

We need nef and log abundant divisors in Theorem \ref{b-thm1.4}. 

\begin{defn}[Nef and abundant line bundles]\label{c-def2.10}
Let $\pi\colon X\to Y$ be a projective 
surjective morphism from a normal complex variety $X$ onto 
a complex variety $Y$. 
Let $\mathcal L$ be a $\pi$-nef 
line bundle on $X$. 
If $\kappa (F, \mathcal L|_F)=\nu (F, \mathcal L|_F)$ holds 
for analytically sufficiently general fibers $F$, 
then $\mathcal L$ is said to be {\em{$\pi$-nef 
and $\pi$-abundant over $Y$}}. 
Similarly, we can define $\pi$-nef and $\pi$-abundant 
$\mathbb Q$-Cartier $\mathbb Q$-divisors. 
\end{defn}

\begin{rem}\label{c-rem2.11} 
In Definition \ref{c-def2.10}, 
if $\mathcal L$ is $\pi$-semiample, 
then it is easy to see that $\mathcal L$ is $\pi$-nef and 
$\pi$-abundant over $Y$. 
\end{rem}

We will freely use the following elementary lemma. 

\begin{lem}\label{c-lem2.12}
Let $\pi\colon X\to Y$ be a projective surjective morphism 
from a normal complex variety $X$ onto a complex variety $Y$ and 
let $\mathcal L$ be a $\pi$-nef and $\pi$-abundant 
line bundle on $X$. 
Let $p\colon Z\to X$ be a projective surjective morphism 
from a normal complex variety $Z$. 
Then $p^*\mathcal L$ is $(\pi\circ p)$-nef and 
$(\pi\circ p)$-abundant over $Y$.  
\end{lem}

\begin{defn}[Nef and log abundant line bundles]\label{c-def2.13}
Let $\pi\colon X\to Y$ be a projective morphism of 
complex analytic spaces and let $(X, \Delta)$ be 
a semi-log canonical pair. 
Let $\mathcal L$ be a line bundle on $X$. 
We say that $\mathcal L$ is {\em{$\pi$-nef 
and $\pi$-log abundant with respect to $(X, \Delta)$ over $Y$}} 
if and only if $\mathcal L|_{S^\nu}$ is nef and abundant 
over $\pi(S)$ for every semi-log canonical stratum $S$ of $(X, \Delta)$, 
where $\mathcal L|_{S^\nu}$ denotes the pull-back of $\mathcal L$ to 
the normalization of $S$. 
Similarly, we can define $\pi$-nef and 
$\pi$-log abundant $\mathbb Q$-Cartier 
$\mathbb Q$-divisors with respect to $(X, \Delta)$. 
\end{defn}

For $\mathbb R$-Cartier $\mathbb R$-divisors, we need the following 
definitions. 
In the present paper, we will use $\mathbb R$-Cartier $\mathbb R$-divisors 
only in Theorems \ref{b-thm1.5}, \ref{b-thm1.7}, \ref{b-thm1.8}, 
Corollary \ref{b-cor1.11}, and Section \ref{y-sec6}. 

\begin{defn}[Relatively abundant $\mathbb R$-Cartier 
$\mathbb R$-divisors]\label{c-def2.14}
Let $\pi\colon X\to Y$ be a projective 
morphism from a normal complex variety $X$ onto a complex 
variety $Y$. Let $D$ be an $\mathbb R$-Cartier $\mathbb R$-divisor 
on $X$. 
If $\kappa _\sigma (F, D|_F)=\kappa _\iota(F, D|_F)$ holds 
for analytically sufficiently general fibers 
$F$, 
then $D$ is said to be $\pi$-abundant over $Y$.   
\end{defn}

For the details of $\kappa _\sigma$ and $\kappa _\iota$, 
see 
\cite[Chapter V, \S 2]{nakayama} and 
\cite[Section 2.5]{fujino-foundations}, respectively. 
For the details of abundant divisors, see also 
\cite[Subsection 2.6.~Abundant divisor]{enokizono-hashizume-mmp}. 

\begin{defn}[Nef and log abundant $\mathbb R$-Cartier 
$\mathbb R$-divisors]\label{c-def2.15}
Let $\pi\colon X\to Y$ be a projective morphism 
of complex analytic spaces and let $(X, \Delta)$ be 
a log canonical pair. 
Let $D$ be an $\mathbb R$-Cartier $\mathbb R$-divisor on $X$. 
We say that $D$ is {\em{$\pi$-nef and 
$\pi$-log abundant with respect to $(X, \Delta)$ over $Y$}} 
if and only if $D|_{S^\nu}$ is nef and 
abundant over $\pi(S)$ for every log canonical stratum of 
$(X, \Delta)$, 
where $D|_{S^\nu}$ denotes the pull-back of $D$ to the normalization 
of $S$. 
\end{defn}

\begin{rem}\label{c-rem2.16}
A $\mathbb Q$-Cartier $\mathbb Q$-divisor $D$ is $\pi$-nef and 
$\pi$-log abundant with respect to $(X, \Delta)$ over $Y$ 
in the sense of Definition \ref{c-def2.15} 
if and only if it is $\pi$-nef and $\pi$-log abundant 
with respect to $(X, \Delta)$ over $Y$ in 
the sense of Definition \ref{c-def2.13}. 
\end{rem}

Let us introduce the notion of 
$B$-bimeromorphic maps, which is obviously a generalization of 
the notion of $B$-birational maps. 

\begin{defn}[$B$-bimeromorphic maps]\label{c-def2.17}
Let $\pi\colon X\to Y$ and $\pi'\colon X'\to Y$ be projective morphisms of complex analytic spaces, and let $(X, \Delta)$ and $(X', \Delta')$ be log canonical pairs. 
We say that a bimeromorphic map $f\colon X\dashrightarrow X'$ over $Y$ is {\em{$B$-bimeromorphic over $Y$}} if there exists a commutative diagram
\begin{equation}
\xymatrix{
&Z\ar[dr]^-{\alpha'}\ar[dl]_-\alpha& \\ 
X\ar[dr]_-\pi\ar@{-->}[rr]^-f& &X'\ar[dl]^-{\pi'}\\ 
&Y&
}
\end{equation}
such that $Z$ is a normal complex analytic space, $\alpha$ and $\alpha'$ are proper bimeromorphic morphisms, and the relation
\begin{equation}
\alpha^*(K_X+\Delta)=\alpha'^*(K_{X'}+\Delta')
\end{equation} 
holds. 
Let $m$ be a positive integer such that $m(K_X+\Delta)$ and $m(K_{X'}+\Delta')$ are Cartier divisors. 
Then we have an isomorphism
\begin{equation}
\begin{split}
f^*\colon \pi'_*\mathcal O_{X'}(m(K_{X'}+\Delta')) 
&\overset{\alpha'^*}\longrightarrow \pi'_*\alpha'_*\mathcal O_Z(\alpha'^*(m(K_{X'}+\Delta')))\\
&\simeq \pi_*\alpha_*\mathcal O_Z(\alpha^*(m(K_X+\Delta)))
\overset{(\alpha^*)^{-1}}\longrightarrow \pi_*\mathcal O_X(m(K_X+\Delta)). 
\end{split}
\end{equation}
We set
\begin{equation}
\Bim(X/Y, \Delta):=\{f\mid \text{$f\colon (X, \Delta)\dashrightarrow (X, \Delta)$ is $B$-bimeromorphic over $Y$}\}. 
\end{equation} 
Then it is obvious that $\Bim(X/Y, \Delta)$ has a natural group structure. 

Let $W$ be a compact subset of $Y$. 
We define
\begin{equation}
\Bim (X/Y, \Delta; W):=\left\{ 
g \;\middle|\; 
\begin{array}{l}
\text{$g \in \Bim\left(\pi^{-1}(U_g)/U_g, \Delta|_{\pi^{-1}(U_g)}\right)$ for some} \\
\text{open neighborhood $U_g$ of $W$ in $Y$}
\end{array}
\right\}. 
\end{equation} 
The subscript $g$ in $U_g$ indicates that the open neighborhood $U_g$ of $W$ may depend on $g$. 
We sometimes consider a semianalytic Stein open subset $U$ of $Y$ satisfying $U \subset W$. 
In such a situation, the chain of inclusions $U \subset W \subset U_g$ always holds. 
Note that $\Bim (X/Y, \Delta; W)$ also possesses a natural group structure. 
\end{defn}

We make small remarks on Definition \ref{c-def2.17}. 

\begin{rem}\label{c-rem2.18}
If $Y$ is a point in Definition \ref{c-def2.17}, 
then $(X, \Delta)$ is a 
projective log canonical pair and 
$\Bim(X/Y, \Delta)$ is nothing but $\Bir(X, \Delta)$ in 
\cite{fujino-abundance} and \cite{fujino-gongyo}. 
\end{rem}

\begin{rem}\label{c-rem2.19}
In Definition \ref{c-def2.17}, 
$X$ and $X'$ are not necessarily irreducible. 
In the proof of Theorem \ref{b-thm1.2}, we have to 
treat $\Bir(X, \Delta)$ in the case where 
$X$ is a disjoint union of normal projective varieties.  
\end{rem}

\begin{rem}\label{c-rem2.20}
Let $(X, \Delta)=:\bigsqcup_i(X_i, \Delta_i)$ and 
$(X', \Delta')=:\bigsqcup _i (X'_i, \Delta'_i)$ be 
the irreducible decompositions. 
Let $f\colon X\dashrightarrow X'$ be a $B$-bimeromorphic 
map over $Y$ as in Definition \ref{c-def2.17}. 
Then, there exists a permutation $\sigma$ such that 
\begin{equation}
f_i:=f|_{X_i}\colon X_i\dashrightarrow X'_{\sigma(i)}
\end{equation} 
is a 
$B$-bimeromorphic map over $Y$ between 
irreducible log canonical pairs 
$(X_i, \Delta_i)$ and $(X'_{\sigma(i)}, 
\Delta'_{\sigma(i)})$. 
We note that $\pi(X_i)=\pi'(X'_{\sigma(i)})$ holds for every $i$. 
\end{rem}

\begin{rem}[{see \cite[Remark 2.15]{fujino-gongyo}}]\label{c-rem2.21}
Let $(X, \Delta)$ and $(X', \Delta')$ be 
log canonical pairs. 
Let $f\colon (X, \Delta)\dashrightarrow (X', \Delta')$ 
be a $B$-bimeromorphic map over $Y$ as in Definition 
\ref{c-def2.17}. We assume that $(X, \Delta-\lfloor 
\Delta\rfloor)$ and $(X', \Delta'-\lfloor \Delta'\rfloor)$ are kawamata 
log terminal. We put $S:=\lfloor \Delta\rfloor$ and $S':=\lfloor 
\Delta'\rfloor$. 
By replacing $Y$ with a relatively compact open subset, 
we may assume that $Z$ in Definition \ref{c-def2.17} is 
smooth and 
\begin{equation} 
\alpha^*(K_X+\Delta)=:K_Z+\Delta_Z:=
\alpha'^*(K_{X'}+\Delta')
\end{equation} 
such that $\Supp \Delta_Z$ is a simple normal crossing 
divisor on $Z$. We may further assume that $\alpha$ and $\alpha'$ are 
projective in Definition \ref{c-def2.17}. We put 
$K_S+\Delta_S:=(K_X+\Delta)|_S$ and 
$K_{S'}+\Delta_{S'}:=(K_{X'}+\Delta')|_{S'}$. 
By applying $\alpha_*$ and $\alpha'_*$ to 
\begin{equation}
0\to \mathcal O_Z(\lceil -(\Delta^{<1}_Z)\rceil-\Delta^{=1}_Z)\to 
\mathcal O_Z(\lceil -(\Delta^{<1}_Z)\rceil) \to 
\mathcal O_{\Delta^{=1}_Z} (\lceil -(\Delta^{<1}_Z)\rceil) \to 0,  
\end{equation}
we have $\alpha_*\mathcal O_{\Delta^{=1}_Z}\simeq \mathcal O_S$ and 
$\alpha'_*\mathcal O_{\Delta^{=1}_Z}\simeq \mathcal O_{S'}$. 
Here we used 
\begin{equation}
R^1\alpha_*\mathcal O_Z(\lceil -(\Delta^{<1}_Z)\rceil-\Delta^{=1}_Z)
=R^1\alpha'_*\mathcal O_Z(\lceil -(\Delta^{<1}_Z)\rceil-\Delta^{=1}_Z)=0, 
\end{equation}  
which is nothing but the relative Kawamata--Viehweg vanishing theorem. 
Thus $f$ induces an isomorphism 
\begin{equation}
(\alpha^*)^{-1}\circ (\alpha')^*\colon 
\pi'_*\mathcal O_{S'}(m(K_{S'}+\Delta_{S'}))\overset{\sim}\longrightarrow 
\pi_*\mathcal O_S(m(K_S+\Delta_S)). 
\end{equation}
We note that $f$ does not necessarily induce a bimeromorphic 
map $S\dashrightarrow S'$ in the above setting. 
\end{rem}

Let us introduce the notion of $B$-pluricanonical representations in 
the relative complex analytic setting.  

\begin{defn}[$B$-pluricanonical representations]\label{c-def2.22}
Let $X$ be a normal complex analytic space 
such that $(X, \Delta)$ is log canonical and let $\pi\colon X\to Y$ be 
a projective morphism of complex analytic spaces. 
Let $m$ be a positive integer such that $m(K_X+\Delta)$ is 
Cartier. Then we have a group homomorphism 
\begin{equation}
\rho_m\colon \Bim(X/Y, \Delta)\to \Aut_{\mathcal O_Y}
\left(\pi_*\mathcal O_X(m(K_X+\Delta))\right) 
\end{equation} 
given by $\rho_m(g)=g^*$ for $g\in \Bim(X/Y, \Delta)$. 
It is called the {\em{$B$-pluricanonical 
representation}} or {\em{log pluricanonical representation}} 
for $(X, \Delta)$ over $Y$. 
When $Y$ is a point, we have 
\begin{equation}
\rho_m\colon \Bir(X, \Delta)\to \Aut_{\mathbb C} \left(H^0(X, 
\mathcal O_X(m(K_X+\Delta)))\right). 
\end{equation}
\end{defn}

Theorem \ref{b-thm1.2} is a generalization of the following 
theorem, 
which is one of the main results of 
\cite{fujino-gongyo}. We note that 
we need it in the proof of Theorem \ref{b-thm1.2}. 
In \cite{hacon-xu}, Hacon and Xu independently 
proved a slightly weaker theorem (see \cite[Theorem 1]{hacon-xu}), 
which seems to be insufficient for the purpose of the present paper. 

\begin{thm}[{\cite[Theorem 1.1]{fujino-gongyo}}]\label{c-thm2.23}
Let $(X, \Delta)$ be a projective log canonical pair. 
Suppose that $m(K_X+\Delta)$ is Cartier and 
that $K_X+\Delta$ is semiample. 
Then $\rho_m\left(\Bir(X, \Delta)\right)$ is a finite group. 
\end{thm}

In the proof of Theorem \ref{b-thm1.2}, Burnside's theorem 
plays a crucial role. 
Hence we state it explicitly for the sake of completeness. 
For the proof, see, for example, \cite[(36.1) Theorem]{curtis-reiner}. 

\begin{thm}[Burnside]\label{c-thm2.24} 
Let $G$ be a subgroup of $\GL(n, \mathbb C)$. 
If the order of any element $g$ of $G$ is uniformly 
bounded, then 
$G$ is a finite group. 
\end{thm}

In order to prove Theorem \ref{b-thm1.1}, 
we need the notion of {\em{admissible}} 
and {\em{preadmissible}} sections, 
which are first introduced in \cite{fujino-abundance}. 

\begin{defn}[{Admissible and preadmissible sections, see 
\cite[Definition 4.1]{fujino-abundance}}]\label{c-def2.25}
Let $(X, \Delta)$ be a semi-divisorial log terminal pair 
and let $\pi\colon X\to Y$ be a projective 
morphism of complex analytic spaces. 
Let $W$ be a compact subset of $Y$. 
Let $X=\bigcup _i X_i$ be the irreducible decomposition. 
As usual, 
\begin{equation}
\nu\colon X^\nu=\bigsqcup _i X_i\to \bigcup _i X_i=X
\end{equation} 
is the normalization with 
\begin{equation}
\nu^*(K_X+\Delta)=K_{X^\nu}+\Theta=:\bigsqcup _i (K_{X_i}+\Theta_i). 
\end{equation}
Let $S$ be the disjoint union of $\lfloor \Theta_i\rfloor$'s. 
We put 
\begin{equation}
K_S+\Delta_S:=(K_{X^\nu}+\Theta)|_S. 
\end{equation} 
Then, by adjunction, $(S, \Delta_S)$ is semi-divisorial 
log terminal. 
Let $m$ be a positive integer such that 
$m(K_X+\Delta)$ is Cartier. 
Let $U$ be a semianalytic Stein open subset of $Y$ with $U\subset W$. 
In particular, the number of the connected components of $U$ is finite 
(see, for example, \cite[Corollary 2.7]{bierstone-milman-semi}). 
We put $X_U:=\pi^{-1}(U)$ and $S_U:=S\cap (\pi\circ \nu)^{-1}(U)$. 
Then we define {\em{preadmissible}} and {\em{admissible}} sections inductively 
as follows. 
\begin{itemize}
\item[(1)] $s\in H^0(X_U, \mathcal O_X 
(m(K_X+\Delta)))\simeq 
H^0(U, \pi_*\mathcal O_X(m(K_X+\Delta)))$ is 
{\em{preadmissible}} if the restriction 
$\nu^*s|_{S_U}\in H^0(S_U, 
\mathcal O_S(m(K_S+\Delta_S)))$ 
is admissible. 
\item[(2)] $s\in H^0(X_U, \mathcal O_X 
(m(K_X+\Delta)))$ is {\em{admissible}} if 
$s$ is preadmissible and $g^*(s|_{X_j})=s|_{X_i}$ holds 
for every $B$-bimeromorphic map $g\colon (X_i, \Theta_i)\dashrightarrow 
(X_j, \Theta_j)$ defined over 
some open neighborhood $U_g$ of $W$ for every $i$, $j$. 
\end{itemize}
Then we put 
\begin{equation}
\begin{split}
&\PA\left(X_U, \mathcal O_X
(m(K_X+\Delta))\right)\\ 
&:=\{s\mid {\text{$s\in H^0(X_U, 
\mathcal O_X 
(m(K_X+\Delta)))$ is preadmissible}}\}
\end{split}
\end{equation} 
and 
\begin{equation}
\begin{split}
&\A\left(X_U, \mathcal O_X
(m(K_X+\Delta))\right)\\ 
&:=\{s\mid {\text{$s\in H^0(X_U, 
\mathcal O_X 
(m(K_X+\Delta)))$ is admissible}}\}. 
\end{split}
\end{equation} 
We note that if $Z$ is any analytic subset 
defined over some open neighborhood 
of $W$ then $U\cap Z$ is a semianalytic Stein open subset 
of $Z$ contained in $W\cap Z$. 
Thus the number of the connected components 
of $U\cap Z$ 
is finite (see, for example, \cite[Corollary 2.7]{bierstone-milman-semi}). 

Let $U'$ be a semianalytic Stein open subset of $Y$ such that 
$U'\subset U$. We put $X_{U'}:=\pi^{-1}(U')$. 
Then there exist natural restriction maps 
\begin{equation}
\PA\left(X_U, \mathcal O_X
(m(K_X+\Delta))\right)
\to \PA\left(X_{U'}, \mathcal O_X
(m(K_X+\Delta))\right)
\end{equation} 
and 
\begin{equation}
\A\left(X_U, \mathcal O_X
(m(K_X+\Delta))\right)
\to \A\left(X_{U'}, \mathcal O_X
(m(K_X+\Delta))\right). 
\end{equation}
\end{defn}

\begin{rem}\label{c-rem2.26} 
In Definition \ref{c-def2.25}, the natural map 
\begin{equation}
H^0(X_U, \mathcal O_X 
(m(K_X+\Delta)))\to  
H^0(U, \pi_*\mathcal O_X(m(K_X+\Delta)))
\end{equation} 
is an isomorphism of topological vector spaces since 
$U$ is Stein (see, for example, \cite[Lemma II.1]{prill}). 
\end{rem}

The following remark is almost obvious by definition. 
We state it explicitly for the sake of completeness. 

\begin{rem}\label{c-rem2.27}
In Definition \ref{c-def2.25}, if 
\begin{equation}
s\in \A\left(X_U, \mathcal O_X
(m(K_X+\Delta))\right) \quad {\text{(resp.~$\PA\left(X_U, \mathcal O_X
(m(K_X+\Delta))\right)$}}, 
\end{equation} then 
\begin{equation}
s^l\in \A\left(X_U, \mathcal O_X
(lm(K_X+\Delta))\right) 
\quad {\text{(resp.~$\PA\left(X_U, \mathcal O_X
(lm(K_X+\Delta))\right)$}}
\end{equation} 
for every positive integer $l$. 
Moreover, if $\A\left(X_U, \mathcal O_X
(m(K_X+\Delta))\right)$ generates $\mathcal O_X(m(K_X+\Delta))$ 
over $U$, then $\A\left(X_U, \mathcal O_X
(lm(K_X+\Delta))\right)$ generates $\mathcal O_X(lm(K_X+\Delta))$ 
over $U$ for every positive integer $l$. 
Similarly, if $\PA\left(X_U, \mathcal O_X
(m(K_X+\Delta))\right)$ generates $\mathcal O_X(m(K_X+\Delta))$ 
over $U$, then $\PA\left(X_U, \mathcal O_X
(lm(K_X+\Delta))\right)$ generates $\mathcal O_X(lm(K_X+\Delta))$ 
over $U$ for every positive integer $l$. 
\end{rem}

The following remark is obvious by definition. 

\begin{rem}\label{c-rem2.28}
In Definition \ref{c-def2.25}, 
if $(X, \Delta)$ is kawamata log terminal, 
then any section $s\in H^0(X_U, \mathcal O_X(m(K_X+\Delta)))$ 
is preadmissible by definition. 
\end{rem}

The following lemma is needed for the proof of Theorem \ref{b-thm1.1}.

\begin{lem}\label{c-lem2.29}
Let $\pi\colon X\to Y$ be a projective morphism of complex analytic spaces and let $(X, \Delta)$ be a semi-log canonical pair such that $m(K_X+\Delta)$ is Cartier. 
Let $W$ be a Stein compact subset of $Y$ such that $\Gamma (W, \mathcal O_Y)$ is noetherian, and let $U$ be a semianalytic Stein open subset of $Y$ with $U\subset W$. 
Let $\nu\colon \bar X\to X$ be the normalization and let $\varphi\colon X'\to \bar X$ be a projective bimeromorphic morphism such that $K_{X'}+\Delta':=(\nu\circ \varphi)^*(K_X+\Delta)$ and $(X', \Delta')$ is divisorial log terminal. 
Let $s'$ be a preadmissible section of $\mathcal O_{X'}(m(K_{X'}+\Delta'))$ on $X'_U:=(\pi\circ \nu \circ \varphi)^{-1}(U)$. 
Then $s'$ descends to a section $s$ of $\mathcal O_X(m(K_X+\Delta))$ on $X_U:=\pi^{-1}(U)$, that is, $s'=(\nu\circ \varphi)^*s$. 
\end{lem}

\begin{proof}
We first note that $X$ has normal crossings in codimension one and satisfies Serre's $S_2$ condition since $(X, \Delta)$ is semi-log canonical. 
We consider the conductor ideal  
\[
\mathfrak{cond}_X:=\mathcal Hom _X(\nu_*\mathcal O_{\bar X}, \mathcal O_X)\subset \mathcal O_X, 
\] 
which is the largest ideal sheaf on $X$ that can also be viewed as an ideal sheaf on $\bar X$. When viewed as an ideal sheaf on $\bar X$, we denote it by $\mathfrak{cond}_{\bar X}$. 
Let $D$ (resp.~$\bar D$) be the closed analytic subspace of $X$ (resp.~$\bar X$) defined by $\mathfrak{cond}_X$ (resp.~$\mathfrak{cond}_{\bar X}$). 
Let $\bar D^\nu \to \bar D$ and $D^\nu \to D$ be the normalizations. 
Then $\bar D^\nu \to D^\nu$ is a finite morphism of normal complex analytic spaces of degree two over every irreducible component. 
Thus, it defines a Galois involution $\tau\colon \bar D^\nu \to \bar D^\nu$ over $D^\nu$. 
More precisely, it is easy to see that a Galois involution naturally exists over a suitable Zariski open subset of $D^\nu$, which can be extended to the aforementioned Galois involution on $\bar D^\nu$ (see, for example, \cite[Lemma 2.24]{enokizono-hashizume-ss}). 
We note that the fixed point set of $\tau$ does not contain any irreducible components of $\bar D^\nu$. 
We have 
\[
K_{\bar X}+ \bar D+\bar \Delta=\nu^*(K_X+\Delta),  
\] 
where $\bar \Delta$ is the preimage of $\Delta$. 
Then we have  
\begin{equation}\label{c-eq2.1}
K_{X'}+\Delta'=\varphi^*(K_{\bar X}+ \bar D+\bar \Delta). 
\end{equation} 
We put 
\[
K_{\bar D^\nu} +\Delta_{\bar D^\nu} :=\nu^*_{\bar D}
(K_{\bar X}+ \bar D+\bar \Delta), 
\] 
where $\nu_{\bar D}\colon \bar D^\nu \to \bar X$ is the induced morphism. 
Then we can easily check that the different $\Delta_{\bar D^\nu}$ is $\tau$-invariant since $\nu\circ \nu_{\bar D}\colon \bar D^\nu\to X$ coincides with $\nu\circ \nu_{\bar D} \circ \tau\colon \bar D^\nu\to X$ (see, for example, \cite[Proposition 5.12]{kollar}). 

Since $\varphi$ is a projective bimeromorphic morphism, $s'$ induces a section $\bar s$ of $\mathcal O_{\bar X}(m(K_{\bar X}+ \bar D+\bar \Delta))$ on $\bar X_U:=(\pi\circ \nu)^{-1}(U)$ with $s'=\varphi^*\bar s$ by \eqref{c-eq2.1}. 
The pull-back of $\bar s$ to $\bar D^\nu$ is a section of $\mathcal O_{\bar D^\nu}(m(K_{\bar D^\nu}+\Delta_{\bar D^\nu}))$, and is $\tau$-invariant since $s'$ is preadmissible. 
Therefore, $\bar s$ descends to a section $s$ on $X_U$ as desired. 
\end{proof}

In our complex analytic setting, 
we can reformulate Claim $(\text{A}_n)$ and Claim $(\text{B}_n)$ in the 
proof of \cite[Lemma 4.9]{fujino-abundance} as follows. 
We note that $(X, \Delta_X)$ is 
{\em{sub log canonical}} when $X$ is smooth and $\Delta_X$ is 
a subboundary $\mathbb Q$-divisor such that $\Supp \Delta_X$ is 
a simple normal crossing divisor. For sub log canonical pairs, 
we can define log canonical centers as in Definition \ref{c-def2.2}. 

\begin{lem}\label{c-lem2.30} 
Let $p \colon Z\to X$ be a projective bimeromorphic 
morphism of smooth 
complex varieties and let $\pi\colon X\to Y$ be a projective 
morphism of complex varieties. 
Let $W$ be a compact subset of $Y$. 
Let $\Delta_Z$ {\em{(}}resp.~$\Delta_X${\em{)}} 
be a subboundary $\mathbb Q$-divisor 
on $Z$ {\em{(}}resp.~$X${\em{)}} such that 
$\Supp \Delta_Z$ {\em{(}}resp.~$\Supp \Delta_X${\em{)}} is a simple 
normal crossing divisor on $Z$ {\em{(}}resp.~$X${\em{)}}. 
We assume that $K_Z+\Delta_Z=p^*(K_X+\Delta_X)$. 
Let $m$ be a positive integer such that 
$m(K_X+\Delta_X)$ is Cartier. Then the 
following statements hold over some open 
neighborhood of $W$. 
\begin{itemize}
\item[{\em{(i)}}] If $T$ is a log canonical center of $(X, \Delta_X)$, 
then there exists a log canonical center $S$ of $(Z, \Delta_Z)$ such that 
$p\colon S\to T$ is bimeromorphic. 
In particular, 
\[p_*\mathcal O_S(m(K_S+\Delta_S))\simeq 
\mathcal O_T(m(K_T+\Delta_T)),
\] where 
$K_S+\Delta_S:=(K_Z+\Delta_Z)|_S$ and $K_T+\Delta_T:=(K_X+\Delta_X)|_T$ by 
adjunction. 
\item[{\em{(ii)}}]  
If $S$ is a log canonical center of $(Z, \Delta_Z)$ such that 
$p\colon S\to p(S)=:T$ is not bimeromorphic, 
then there exists a log canonical center $S'$ of $(Z, \Delta_Z)$ 
with $S'\subset S$ such that 
$p\colon S'\to T$ is bimeromorphic and 
the restriction map 
\begin{equation}
p_*\mathcal O_S(m(K_S+\Delta_S))\to p_*\mathcal O_{S'}
(m(K_{S'}+\Delta_{S'})),  
\end{equation} 
induced by the inclusion $S'\hookrightarrow S$ and 
adjunction, is an isomorphism, 
where 
$K_{S'}+\Delta_{S'}:=(K_Z+\Delta_Z)|_{S'}$ by adjunction. 
We note that 
\[p_*\mathcal O_{S'}(m(K_{S'}+\Delta_{S'}))\simeq 
\mathcal O_T(m(K_T+\Delta_T))
\] 
obviously holds.  
\end{itemize}
\end{lem}

\begin{proof}[Sketch of Proof of Lemma \ref{c-lem2.30}] 
With suitable modifications, 
the proofs of Claims $(\text{A}_n)$ and $(\text{B}_n)$ 
in the proof of \cite[Lemma 4.9]{fujino-abundance} 
also work in our setting (see also \cite[Lemma 7.2]{fujino-phd}). 
Therefore, we only sketch the argument here.

We can verify (i) by induction on the dimension of $X$. 
To prove (ii), by blowing up $Z$ along the center $S$, 
we may assume that $S$ is a divisor on $Z$. 
Using (i), we can reduce the problem to the case 
where $p\colon Z \to X$ is a finite composition of 
blow-ups with centers corresponding to $S$. 
The statement then follows by a direct check.
\end{proof}

We will freely use Lemma \ref{c-lem2.30} in subsequent sections. 

\section{Finiteness of relative log pluricanonical representations}\label{a-sec3}

In this section, we prove Theorem \ref{b-thm1.2} and
Corollary \ref{b-cor1.3}.
We note that the proof of Theorem \ref{b-thm1.2} relies on Theorem \ref{c-thm2.23}.
We begin with an elementary lemma.

\begin{lem}\label{a-lem3.1}
Let $Y$ be a connected complex manifold. 
Let 
\begin{equation}
\rho \colon G\to \GL(r, \mathcal O_Y)
\end{equation} 
be a group homomorphism. 
We further consider 
\begin{equation}
\rho_y:=\ev_y\circ \rho \colon G\to 
\GL(r, \mathcal O_Y)\to \GL(r, \mathbb C), 
\end{equation} 
where $\ev_y$ is the evaluation map at $y\in Y$. 
We assume that $\xIm \rho_y=\rho_y(G)$ is a finite 
group for every $y\in Y$. 
Then 
$
\ev_y\colon \rho(G)\to\rho_y(G)
$ 
is an isomorphism for every $y\in Y$. 
In particular, $\xIm\rho =\rho(G)$ is a finite group. 
\end{lem}

\begin{proof}
It is obvious that $\ev_y\colon \rho(G)\to \rho_y(G)$ is surjective 
for every $y\in Y$. 
We take an arbitrary point $y_0\in Y$. 
It suffices to prove that 
$\ev_{y_0}\colon \rho (G)\to \rho_{y_0}(G)$ is injective. 
We take $g\in \rho (G)$ such that 
$\ev_{y_0}(g)=E_r$, where $E_r$ is the $r\times r$ identity 
matrix. 
Note that $\ev_y(g)$ is semisimple and 
every eigenvalue of $\ev_y(g)$ is a root of 
unity for every $y\in Y$ since $\rho_y(G)$ is a 
finite group by assumption. 
We consider the characteristic polynomial 
$\chi(t):=\det (tE_r-g)$. 
The coefficients of $\chi(t)$ are holomorphic 
and take values in $K$, 
where $K$ is the subfield of $\mathbb C$ generated by 
all roots of unity. 
Hence they are constant.  
Since $\ev_{y_0}(g)=E_r$, 
we see that every eigenvalue of $\ev_y(g)$ is $1$ for every $y\in Y$. 
This implies that $\ev_y(g)=E_r$ holds for every 
$y\in Y$ because 
$\ev_y(g)$ is semisimple. 
Hence we have $g=E_r$, that is, $\ev_{y_0}\colon 
\rho(G)\to \rho_{y_0}(G)$ is injective. 
We finish the proof. 
\end{proof}

Theorem \ref{a-thm3.2} is one of the most important 
results in the present paper. 

\begin{thm}\label{a-thm3.2}
Let $\pi\colon X\to Y$ be a projective morphism 
from a normal complex analytic space $X$ onto a polydisc 
$Y$ such that $(X, \Delta)$ is divisorial log terminal and 
that $K_X+\Delta$ is $\pi$-semiample. 
Let $\varphi\colon Z\to X$ be a projective bimeromorphic 
morphism from a smooth complex analytic space $Z$ with 
$K_Z+\Delta_Z:=\varphi^*(K_X+\Delta)$ such that 
$\pi\circ \varphi\colon Z\to Y$ is smooth and projective,  
and 
that $\Exc(\varphi)$ and $\Exc(\varphi)\cup\Supp \varphi^{-1}_*\Delta$ are simple normal crossing 
divisors on $Z$ and are relatively normal crossing 
over $Y$. 
Let $m$ be a positive integer such that $m(K_X+\Delta)$ 
is Cartier. 
We assume that \[R^i\pi_*\mathcal O_X(m(K_X+\Delta))\] is locally 
free for every $i$ and 
\[\pi_*\mathcal O_X(m(K_X+\Delta))\simeq \mathcal O_Y^{\oplus r}\] 
for some positive integer $r$. 
We consider 
\begin{equation}
\rho_m\colon \Bim(X/Y, \Delta)\to 
\GL(r, \mathcal O_Y)\simeq \Aut_{\mathcal O_Y} 
\left(\pi_*\mathcal O_X(m(K_X+\Delta))\right)
\end{equation}
and 
\begin{equation}
\rho_{m, y}:=\ev_y\circ \rho_m\colon 
\Bim(X/Y, \Delta)\to 
\GL(r, \mathcal O_Y)\to \GL(r, \mathbb C), 
\end{equation} 
where $\ev_y$ is the evaluation map at $y\in Y$. 
Then $\xIm \rho_{m, y}$ is a finite group for every $y\in Y$. 
Moreover, 
\begin{equation}
\ev_y\colon \xIm \rho_m \to \xIm \rho_{m, y}
\end{equation} 
is an isomorphism for every $y\in Y$. 
In particular, $\xIm \rho_m$ is a finite group. 

We note that, in the above setting, 
$X_y:=\pi^{-1}(y)$ is a normal projective scheme, 
$(X_y, \Delta_y)$ is divisorial log terminal, 
where $K_{X_y}+\Delta_y:=(K_X+\Delta)|_{X_y}$, and 
\begin{equation}\label{a-eq3.1}
\ev_y\colon \pi_*\mathcal O_X(m(K_X+\Delta))\to 
H^0(X_y, \mathcal O_{X_y}(m(K_{X_y}+\Delta_y))) 
\end{equation} 
by the base change theorem. 
\end{thm}

\begin{rem}\label{a-rem3.3}
In Theorem \ref{a-thm3.2}, $X$ is not necessarily connected. 
\end{rem}

Let us prove Theorem \ref{a-thm3.2}. 

\begin{proof}[Proof of Theorem \ref{a-thm3.2}]
By Lemma \ref{a-lem3.1}, it suffices to 
prove the finiteness of $\xIm \rho _{m, y}$ for every $y\in Y$. 
In Step \ref{a-thm3.2-step1}, we will prove the description \eqref{a-eq3.1} 
of the evaluation map $\ev_y$. 
Then, in Step \ref{a-thm3.2-step2}, 
we will prove the finiteness of $\xIm \rho_{m, y}$. 
\setcounter{step}{0}

\begin{step}\label{a-thm3.2-step1}
We put $d:=\dim Y$. 
We note that $Y$ is a polydisc by assumption. 
We take general hyperplanes $H_1, \dots, H_d$ on $Y$ passing through $y$. 
We note that $(\pi\circ \varphi)^*(\sum _{i=1}^d H_i)$ and $\Exc(\varphi)\cup \Supp \bigl(\varphi^{-1}_*\Delta + (\pi\circ \varphi)^*(\sum _{i=1}^d H_i)\bigr)$ are simple normal crossing divisors on $Z$. 
By construction, \[a(E, X, \Delta+\sum _{i=1}^d \pi^*H_i) = a(E, X, \Delta) > -1\] holds for every $\varphi$-exceptional divisor $E$. 
Therefore, $(X, \Delta+\sum _{i=1}^d\pi^*H_i)$ is a divisorial log terminal pair. 
By construction, $X_y$ is a log canonical center of $(X, \Delta +\sum _{i=1}^d \pi^*H_i)$. 
This implies that $X_y$ is normal and $(X_y, \Delta_y)$ is divisorial log terminal. 
Since $R^i\pi_*\mathcal O_X(m(K_X+\Delta))$ is locally free for every $i$ by assumption, we have 
\begin{equation}
\pi_*\mathcal O_X(m(K_X+\Delta))\otimes \mathbb C(y) \simeq H^0(X_y, \mathcal O_{X_y}(m(K_{X_y}+\Delta_y)))
\end{equation} 
by the base change theorem. 
Hence we have the desired description \eqref{a-eq3.1} of the evaluation map $\ev_y$. 
\end{step}

\begin{step}
\label{a-thm3.2-step2}
We take an arbitrary element $g$ of 
$\Bim(X/Y, \Delta)$. 
By Theorem \ref{c-thm2.24}, it suffices to prove that 
the order of $\rho_{m, y}(g)=\ev_{y}\circ \rho_m(g)$ is uniformly 
bounded. We make $H_1$ general in Step \ref{a-thm3.2-step1} and 
put $Y':=H_1$, $X':=\pi^*H_1$, and $K_{X'}+\Delta':=(K_X+X'+\Delta)|_{X'}$. 
Then the above $g$ induces $g'\in \Bim(X'/Y', \Delta')$ 
such that $\ev_{y}\circ \rho_m(g)=\ev_{y}\circ \rho'_m(g')$ holds, 
where 
\begin{equation}
\rho'_m\colon \Bim(X'/Y', \Delta')\to  
\Aut_{\mathcal O_{Y'}}\left(\pi_*\mathcal O_{X'}
(m(K_{X'}+\Delta'))\right). 
\end{equation}
By repeating this process finitely many times, we may assume that 
$Y$ is a disc. Hence $X_y$ is a divisor on $X$. 

We first assume that $X_y$ is connected. 
Let $l$ be the number of the log canonical strata of $(X_y, \Delta_y)$. 
We consider 
\begin{equation}
\rho_m\colon \Bir(V, \Delta_V)\to \Aut_{\mathbb C} 
\left(H^0(V, \mathcal O_V(m(K_V+\Delta_V)))\right), 
\end{equation} 
where $(V, \Delta_V)$ is a log canonical stratum of $(X_y, \Delta_y)$. 
Since $K_V+\Delta_V$ is semiample, $\rho_m \left(\Bir(V, \Delta_V)\right)$ 
is a finite group by Theorem \ref{c-thm2.23}. 
Then we put 
\begin{equation}
k:=\lcm \left\{ \# \rho_m \left(\Bir(V, \Delta_V)\right) \mid 
\text{$(V, \Delta_V)$ is a log canonical stratum of $(X_y, \Delta_y)$} 
\right\}
\end{equation}
\begin{claim}\label{a-thm3.2-claim}
$\rho_{m, y}(g)^{l!k}=E_r$ holds. 
\end{claim}
\begin{proof}[Proof of Claim]
We consider log canonical strata $(T, \Delta_T)$ 
of $(X_y, \Delta_y)$ satisfying 
that the natural restriction map 
\begin{equation}\label{a-eq3.2}
H^0(X_y, \mathcal O_{X_y}(m(K_{X_y}+\Delta_y)))
\to H^0(T, \mathcal O_T(m(K_T+\Delta_T)))
\end{equation} 
is an isomorphism. 
We put $t:=\min \dim T$. 

Let $(T, \Delta_T)$ be a $t$-dimensional 
log canonical stratum of $(X_y, \Delta_y)$ such that 
the natural restriction map \eqref{a-eq3.2} is an isomorphism. 
We consider the following commutative diagram as in Definition \ref{c-def2.17}
\begin{equation}
\xymatrix{
&X^\dag\ar[dr]^-{\beta}\ar[dl]_-\alpha& 
\\ 
X\ar[dr]_-\pi\ar@{-->}[rr]^-g& &X\ar[dl]^-{\pi}\\ 
&Y&
}
\end{equation}
where $g$ is a $B$-bimeromorphic map of $(X, \Delta)$ 
over $Y$ taken above. By shrinking $Y$ around $y$, 
we may assume that 
$X^\dag$ is smooth, 
$\alpha$ and $\beta$ are projective, and 
\begin{equation*}
\alpha^*(K_X+\Delta)=:K_{X^\dag}+\Delta_{X^\dag}:=\beta^*(K_X+\Delta)
\end{equation*} 
such 
that $\Supp \Delta_{X^\dag}\cup 
\Supp (\pi\circ \alpha)^*y$ is a simple normal crossing 
divisor on $X^\dag$. 
We take $X^\dag$ suitably. 
Then, by Lemma \ref{c-lem2.30} (see also 
the proof of \cite[Lemma 4.9]{fujino-abundance} 
and \cite[Lemma 2.16]{fujino-gongyo}), we can find a log canonical 
stratum $(T', \Delta_{T'})$ of $(X_y, \Delta_y)$ and 
a commutative diagram 
\begin{equation}
\xymatrix{
& T^\dag\ar[dr]^-{\beta|_{T^\dag}}\ar[dl]_-{\alpha|_{T^\dag}}& \\ 
(T, \Delta_T) & & (T', \Delta_{T'})
}
\end{equation} 
such that $\alpha|_{T^\dag}$ and 
$\beta|_{T^\dag}$ are proper birational and 
that 
\begin{equation}
(\beta|_{T^\dag})\circ (\alpha|_{T^\dag})^{-1}
\colon (T, \Delta_T)\dashrightarrow (T', \Delta_{T'})
\end{equation} 
is 
a $B$-birational map of projective divisorial log terminal pairs. 
Note that there are only finitely many log canonical strata contained 
in $X_y$. 
Thus we can find $t$-dimensional log canonical strata 
$(S_i, \Delta_{S_i})$ of $(X_y, \Delta_y)$ for 
$1\leq i\leq p$ and 
a natural embedding 
\begin{equation}
H^0(X_y, \mathcal O_{X_y}(m(K_{X_y}+\Delta_y)))
\hookrightarrow \bigoplus _i H^0(S_i, \mathcal O_{S_i}
(m(K_{S_i}+\Delta_{S_i})))
\end{equation}
such that $g$ induces $\widetilde g\in \Bir(S, \Delta_S)$, 
where $(S, \Delta_S):=\bigsqcup_i (S_i, \Delta_{S_i})$, satisfying 
the following commutative diagram: 
\begin{equation}
\xymatrix{
0\ar[r] &\ar[d]_-{\rho_{m, y} (g)}\ar[r] 
H^0(X_y, \mathcal O_{X_y}(m(K_{X_y}+\Delta_y)))
& \bigoplus _i H^0(S_i, \mathcal O_{S_i}
(m(K_{S_i}+\Delta_{S_i})))\ar[d]^-{\rho_m(\widetilde g)}\\ 
0 \ar[r]& 
H^0(X_y, \mathcal O_{X_y}(m(K_{X_y}+\Delta_y)))
\ar[r]& \bigoplus _i H^0(S_i, \mathcal O_{S_i}
(m(K_{S_i}+\Delta_{S_i}))). 
}
\end{equation} 
We note the following description of $\rho_{m, y}(g)$. 
Let $V$ be the union of the irreducible components of 
$(\Delta_{X^\dag}+(\pi\circ \alpha)^*y)^{=1}$ mapped to $y$.   
We put 
\begin{equation*}
K_V+\Delta_V:=(K_{X^\dag}+\Delta_{X^\dag}+(\pi\circ \alpha)^*y)|_V.
\end{equation*}  
Then we can check that 
$\alpha_*\mathcal O_V\simeq \mathcal O_{X_y}\simeq \beta_*\mathcal O_V$ 
holds, which is an 
easy consequence of the strict support condition 
established in \cite[Theorem 1.1 (i)]{fujino-vanishing} 
(see, for example, the proof of Lemma \ref{p-lem4.2} below). 
Thus $\rho_{m, y}(g)$ can be written as 
\begin{equation*}
\begin{split}
\rho_{m, y}(g)\colon H^0(X_y, \mathcal O_{X_y}(m(K_{X_y}+\Delta_y)))
\overset{\beta^*}{\longrightarrow} \
&H^0(V, \mathcal O_V(m(K_V+\Delta_V)))
\\
\overset{(\alpha^*)^{-1}}\longrightarrow  
&H^0(X_y, \mathcal O_{X_y}(m(K_{X_y}+\Delta_y))). 
\end{split}
\end{equation*}
Since $\rho_m(\widetilde g)^{l!k}=\id$ on 
$\bigoplus _i H^0(S_i, \mathcal O_{S_i}
(m(K_{S_i}+\Delta_{S_i})))$ by the definitions of $l$ and $k$, 
we have $\rho_{m, y}(g)^{l!k}=E_r$. 
This completes the proof of the claim.
\end{proof}
We note that $l!k$ is independent of $g$. 
Therefore, the claim implies that $\xIm \rho_{m, y}$, 
which is a subgroup of $\GL(r, \mathbb C)$, is a finite group 
by Burnside's theorem (see Theorem \ref{c-thm2.24}). 
Thus we finish the proof under the assumption that 
$X_y$ is connected. 

From now on, we assume that 
$X_y$ is not connected. Let $a$ denote the number of 
the connected components of $X_y$. 
Then $g^{a!}$ preserves each connected component of $X_y$. 
Thus, by the above argument, we can take a positive integer $b$ such that 
$\rho_{m, y}(g)^b=E_r$ holds for every $g\in \Bim(X/Y, \Delta)$. 
Thus, by Burnside's theorem (see Theorem \ref{c-thm2.24}), 
we see that $\xIm\rho_{m, y}$ is a finite group. 
\end{step}
We finish the proof. 
\end{proof}

We can prove Theorem \ref{b-thm1.2} as an easy application of 
Theorem \ref{a-thm3.2}. 

\begin{proof}[Proof of Theorem \ref{b-thm1.2}]
Let $U$ be a nonempty open subset of $Y$. 
We consider the following commutative diagram 
\begin{equation}
\xymatrix{
\rho_m\colon \Bim (X/Y, \Delta)\ar[d]\ar[r]&\Aut_{\mathcal O_Y}\left(
\pi_*\mathcal O_X(m(K_X+\Delta))\right)\ar[d]\\
\rho_m\colon \Bim (\pi^{-1}(U)/U, \Delta|_{\pi^{-1}(U)})
\ar[r]&\Aut_{\mathcal O_U}\left(
\pi_*\mathcal O_{\pi^{-1}(U)}(m(K_X+\Delta))\right). 
}
\end{equation}
Note that the vertical arrows are natural restriction maps. 
It is obvious that 
the restriction map 
\begin{equation}
\Aut_{\mathcal O_Y}\left(
\pi_*\mathcal O_X(m(K_X+\Delta))\right)
\to
\Aut_{\mathcal O_U}\left(
\pi_*\mathcal O_{\pi^{-1}(U)}(m(K_X+\Delta))\right)
\end{equation} 
is injective since $Y$ is irreducible and every irreducible 
component of $X$ is dominant onto $Y$. Hence, in order to 
prove Theorem \ref{b-thm1.2}, we can freely replace $Y$ with 
a small nonempty open subset of $Y$. 
We take a Stein compact subset $W$ of $Y$ such that 
$\Gamma(W, \mathcal O_Y)$ is noetherian. 
Then, by \cite[Theorems 1.21 and 1.27]{fujino-bchm}, we can take a dlt blow-up 
$\psi\colon (X', \Delta')\to (X, \Delta)$. 
By replacing $\pi\colon (X, \Delta)\to Y$ with 
$\pi':=\pi\circ \psi\colon (X', \Delta')\to 
Y$, we may further assume that $(X, \Delta)$ 
is divisorial log terminal. 
By taking a resolution of singularities of $X$ (see, 
for example, \cite{bierstone-milman}) and shrinking 
$Y$ suitably, we may assume that $\pi\colon (X, \Delta)\to Y$ satisfies 
all the conditions in Theorem \ref{a-thm3.2}. 
Then, by Theorem \ref{a-thm3.2}, 
$\rho_m\left(\Bim(X/Y, \Delta)\right)$ is a finite group. 
This is what we wanted. We finish the proof. 
\end{proof}

Let us prove Corollary \ref{b-cor1.3}, which is almost obvious 
by Theorem \ref{b-thm1.2}. We will use it in the proof of Theorem \ref{b-thm1.1}. 

\begin{proof}[Proof of Corollary \ref{b-cor1.3}]
We decompose $(X, \Delta)=:\bigsqcup_i (X_i, \Delta_i)$ 
such that 
$\pi_i:=\pi|_{X_i}\colon X_i\to Y_i:=\pi(X_i)$ is surjective and 
every irreducible component of $X_i$ is dominant onto $Y_i$ for every $i$. 
We may assume that $Y_i\ne Y_j$ for $i\ne j$. 
Since $U$ is a semianalytic Stein open subset of $Y$ with $U\subset W$, 
$Y_i\cap U$ is a finite 
disjoint union of semianalytic Stein open subsets of $Y_i$ 
(see, for example, \cite[Corollary 2.7]{bierstone-milman-semi}). 
Let $U'$ be a connected component of $Y_i\cap U$. 
Then, by Theorem \ref{b-thm1.2}, the image of 
\begin{equation}\label{a-eq3.3}
\rho_m\colon \Bim\left(\pi_i^{-1}(U')/{U'}, \Delta_i|_{\pi_i^{-1}(U')}\right)
\to \Aut _{\mathcal O_{U'}}\left(\pi_{i*}\mathcal O_{\pi_i^{-1}(U')}
(m(K_{X_i}+\Delta_i))\right)
\end{equation} 
is a finite group. 
Note that there exists a natural restriction map 
\begin{equation}\label{a-eq3.4}
\Bim(X/Y, \Delta; W)\to 
\Bim\left(\pi_i^{-1}(U')/{U'}, \Delta_i|_{\pi_i^{-1}(U')}\right). 
\end{equation}
By the natural restriction map \eqref{a-eq3.4}, 
\begin{equation}
\rho^{WU'}_m\colon \Bim(X/Y, \Delta; W)\to 
\Aut_{\mathcal O_{U'}}\left(\pi_{i*}\mathcal O_{\pi_i^{-1}(U')}
(m(K_{X_i}+\Delta_i))\right)
\end{equation} 
factors 
through $\rho_m$ in \eqref{a-eq3.3}. 
Thus, we have 
\begin{equation}
\rho^{WU'}_m(\Bim(X/Y, \Delta; W))\subset 
\rho_m \left(\Bim\left(\pi_i^{-1}(U')/{U'}, 
\Delta_i|_{\pi_i^{-1}(U')}\right)\right). 
\end{equation} 
The group $\rho^{WU}_m(\Bim(X/Y, \Delta; W))$ is contained 
in 
\begin{equation}
\prod _{U'} \rho^{WU'}_m(\Bim(X/Y, \Delta; W)), 
\end{equation}  
where $U'$ runs over all connected components of $Y_i\cap U$ for all $i$. 
Hence, we see that 
$\rho^{WU}_m\left (\Bim(X/Y, \Delta; W)\right)$ is 
a finite group. 
We finish the proof. 
\end{proof}

\section{Abundance for semi-log canonical pairs}\label{p-sec4}

In this section, we prove Theorem \ref{b-thm1.1}.
Our strategy follows that of \cite{fujino-abundance},
except that we make use of the minimal model program for
projective morphisms between complex analytic spaces 
established in \cite{fujino-bchm}, 
\cite{enokizono-hashizume-ss}, and \cite{enokizono-hashizume-mmp}.

The following lemma is well known and follows easily
from the relative Kawamata--Viehweg vanishing theorem.

\begin{lem}[Connectedness lemma]\label{p-lem4.1}
Let $(X, \Delta)$ be a log canonical pair and 
let $\pi\colon X\to Y$ be 
a projective morphism of complex analytic spaces with 
$\pi_*\mathcal O_X\simeq \mathcal O_Y$.  
Assume that $-(K_X+\Delta)$ is $\pi$-nef and 
$\pi$-big. 
Then $\Nklt(X, \Delta)\cap \pi^{-1}(y)$ is connected 
for every $y\in Y$, where $\Nklt(X, \Delta)$ denotes 
the non-klt locus of $(X, \Delta)$. 
In particular, if $(X, \Delta-\lfloor \Delta\rfloor)$ is 
kawamata log terminal, $\lfloor \Delta\rfloor 
\cap \pi^{-1}(y)$ is connected 
for every $y\in Y$.  
\end{lem}

\begin{proof} 
The standard proof in the algebraic setting works with only 
minor modifications. This is because the Kawamata--Viehweg vanishing 
theorem holds for projective morphisms between complex analytic spaces. 
In this proof, we can freely shrink $Y$ around $y$. 
We consider the following short exact sequence: 
\begin{equation}
0\to \mathcal J(X, \Delta)\to \mathcal O_X\to \mathcal O_{\Nklt(X, \Delta)}\to 0, 
\end{equation} 
where $\mathcal J(X, \Delta)$ denotes the multiplier ideal sheaf 
of $(X, \Delta)$. 
By the relative Kawamata--Viehweg--Nadel vanishing 
theorem, 
we have 
\begin{equation}
0\to 
\pi_*\mathcal J(X, \Delta)\to \mathcal O_Y\to 
\pi_*\mathcal O_{\Nklt(X, \Delta)}\to 0. 
\end{equation} 
This implies that $\Nklt(X, \Delta)\cap \pi^{-1}(y)$ is connected. 
\end{proof}

The following lemma also asserts that the union of log canonical centers 
is connected under a suitable setting. Lemma \ref{p-lem4.2} is 
substantially more difficult than Lemma \ref{p-lem4.1}. 
Its proof relies heavily on the strict support condition 
established in \cite[Theorem 1.1 (i)]{fujino-vanishing} (see 
also \cite{fujino-vanishing-pja} and \cite{fujino-fujisawa}).

\begin{lem}\label{p-lem4.2}
Let $(X, \Delta)$ be a log canonical pair and 
let $\pi\colon X\to Y$ be a projective morphism of normal 
complex varieties with $\pi_*\mathcal O_X\simeq 
\mathcal O_Y$. 
Let $W$ be a compact subset of $Y$. 
We assume that $K_X+\Delta\sim_{\mathbb Q, \pi}0$ 
holds. 
We put $Y':=\bigcup _i\pi(C_i)\subsetneq Y$, where 
$\{C_i\}$ is a set of some log canonical centers 
of $(X, \Delta)$. Let $X'$ be the union of the log canonical 
centers of $(X, \Delta)$ mapped to $Y'$ by $\pi$. 
Then, after shrinking $Y$ around $W$ suitably, 
$\pi_*\mathcal O_{X'}\simeq \mathcal O_{Y'}$ holds. 
In particular, $\pi_*\mathcal O_{X'}\simeq \mathcal O_{Y'}$ holds 
on an open subset $U$ contained in $W$. 
\end{lem}

\begin{proof}
Throughout this proof, we will freely shrink $Y$ around $W$ without explicit mention. 
Let $p\colon Z\to X$ be a projective bimeromorphic 
morphism from a smooth complex variety $Z$ with 
$K_Z+\Delta_Z:=p^*(K_X+\Delta)$ (see \cite{bierstone-milman}). 
We may assume that $(\pi\circ p)^{-1}(Y')$ and 
$p^{-1}(X')$ are simple normal crossing divisors on $Z$. 
We may further assume that the union of 
$(\pi\circ p)^{-1}(Y')$, 
$p^{-1}(X')$, and $\Supp \Delta_Z$ is contained in a simple 
normal crossing divisor on $Z$. 
Let $V$ be the union of the irreducible components 
of $\Delta^{=1}_Z$ mapped to $Y'$ by $\pi\circ p$. 
We put $A:=\lceil -(\Delta^{<1}_Z)\rceil$, which is 
a $p$-exceptional effective divisor on $Z$. 
By assumption, we have 
\begin{equation}
A-V-(K_Z+\Delta^{=1}_Z-V+\{\Delta_Z\})\sim 
_{\mathbb Q, \pi\circ p}0. 
\end{equation}
We consider the following short exact sequence 
\begin{equation}
0\to \mathcal O_Z(A-V)\to \mathcal O_Z(A)\to \mathcal O_V(A)\to 0. 
\end{equation} 
We note that no log canonical 
centers of $(Z, \Delta^{=1}_Z-V+\{\Delta_Z\})$ map to 
$Y'$ by construction. 
Then we have 
\begin{equation}
0\to (\pi\circ p)_*\mathcal O_Z(A-V)\to \mathcal O_Y\to 
(\pi\circ p)_*\mathcal O_V(A)\to 0.  
\end{equation}
Here we used the 
strict support condition for 
$R^1(\pi\circ p)_*\mathcal O_Z(A-V)$ (see \cite[Theorem 1.1 (i)]{fujino-vanishing}) 
in order to 
prove that the connecting homomorphism 
\begin{equation}
\delta\colon 
(\pi\circ p)_*\mathcal O_V(A)\to R^1(\pi\circ p)_*\mathcal O_Z(A-V)
\end{equation} 
is zero. 
This implies that $(\pi\circ p)_*\mathcal O_V(A)\simeq 
\mathcal O_{Y'}$ holds. 
Similarly, we have the 
short exact sequence 
\begin{equation}
0\to p_*\mathcal O_Z(A-V)\to \mathcal O_X\to p_*\mathcal O_V(A)\to 0  
\end{equation} 
since no log canonical centers of $(Z, \Delta^{=1}_Z-V+\{\Delta_Z\})$ map to 
$X'$ by construction. 
This implies that $p_*\mathcal O_V(A)\simeq \mathcal O_{X'}$. 
Hence we have $\pi_*\mathcal O_{X'}\simeq \mathcal O_{Y'}$. 
We finish the proof 
of Lemma \ref{p-lem4.2}. 
\end{proof}

As a straightforward corollary of Lemma \ref{p-lem4.2}, we obtain the following:

\begin{cor}[{\cite[Lemma 4.2]{fujino-abundance}}]\label{p-cor4.3}
Let $(X, \Delta)$ be a divisorial log terminal pair and 
let $\pi\colon X\to Y$ be a projective morphism of normal 
complex varieties with $\pi_*\mathcal O_X\simeq 
\mathcal O_Y$. 
Let $W$ be a compact subset of $Y$. 
We assume that $K_X+\Delta\sim_{\mathbb Q, \pi}0$ 
holds. If $Y':=\pi(\lfloor \Delta\rfloor)\subsetneq Y$, 
then, after shrinking $Y$ around $W$ suitably, 
we have $\pi_*\mathcal O_{\lfloor \Delta\rfloor}\simeq 
\mathcal O_{Y'}$. 
\end{cor}

\begin{proof}
Since $(X, \Delta)$ is divisorial log terminal, 
$\lfloor \Delta\rfloor$ is the union of 
all log canonical centers of $(X, \Delta)$. 
Therefore, by Lemma \ref{p-lem4.2}, 
we have 
$\pi_*\mathcal O_{\lfloor \Delta\rfloor}\simeq 
\mathcal O_{Y'}$. 
\end{proof}

In this paper, Lemma \ref{p-lem4.5} plays a crucial role. To facilitate the reader's understanding, we first treat the following lemma, which serves as a toy model for Lemma \ref{p-lem4.5} and Proposition \ref{p-prop4.6} below. We believe that its proof will be helpful for grasping the core ideas of Lemma \ref{p-lem4.5}. Furthermore, Lemma \ref{p-lem4.4}, which is a special case of Lemma \ref{p-lem4.5}, is already sufficient for the purposes of \cite{fujino-phd}, \cite{fujino-index}, \cite{fujino-index-a}, and \cite{gongyo-abundance}. 

\begin{lem}\label{p-lem4.4}
Let $(X, \Delta)$ be a projective $\mathbb Q$-factorial 
divisorial log terminal pair such that $K_X+\Delta\sim _{\mathbb Q}0$. 
Assume that $\lfloor \Delta\rfloor$ is not connected. 
Then $\lfloor \Delta\rfloor =S_1+S_2$ such that 
$(S_i, \Delta_{S_i})$ is kawamata log terminal 
with $K_{S_i}+\Delta_{S_i}:=(K_X+\Delta)|_{S_i}$ for 
$i=1, 2$ and 
that $(S_1, \Delta_{S_1})$ is $B$-birationally equivalent to 
$(S_2, \Delta_{S_2})$. 
In particular, $(X, \Delta)$ is purely log terminal. 
\end{lem}

We note that, by \cite[Theorem 1.2]{gongyo-mmp}, the condition $K_X+\Delta\sim_{\mathbb Q}0$ in Lemma \ref{p-lem4.4} is actually equivalent to the condition that $K_X+\Delta$ is numerically trivial.

\begin{proof}[Proof of Lemma \ref{p-lem4.4}]
Note that $K_X+\Delta-\varepsilon \lfloor \Delta\rfloor$ is not pseudo-effective for any sufficiently small positive rational number $\varepsilon$. 
We run a $(K_X+\Delta-\varepsilon \lfloor \Delta\rfloor)$-minimal model program with ample scaling (see \cite{bchm}). 
We know that the number of connected components of $\lfloor \Delta\rfloor$ is preserved by the above minimal model program, as we apply Lemma \ref{p-lem4.1} at each step. 
Therefore, we finally obtain an extremal Fano contraction morphism which is generically a $\mathbb P^1$-bundle with two disjoint sections. 
More precisely, we have the following diagram:
\begin{equation}
\xymatrix{
p\colon X \ar@{-->}[r] & X' \ar[d]^\varphi \\ 
& V
}
\end{equation}
where $p\colon X\dashrightarrow X'$ is a finite sequence of flips and divisorial contractions, and $\varphi\colon X'\to V$ is a $(K_{X'}+\Delta'-\varepsilon \lfloor \Delta'\rfloor)$-negative extremal Fano contraction with $\Delta':=p_*\Delta$. 
Furthermore, $\dim V=\dim X-1$, $\lfloor \Delta'\rfloor =S'_1+S'_2$, $\varphi\colon S'_i\to V$ is an isomorphism for $i=1, 2$, and $S'_1\cap S'_2=\emptyset$. 

By cutting $V$ with general hyperplanes $\dim V - 1$ times to reduce to the case where $V$ is a curve, and applying \cite[12.3.4 Theorem]{flips-and-abundance}, we see that there exists an effective $\mathbb Q$-divisor $P$ on $V$ such that $\varphi\colon (S'_i, \Delta_{S'_i})\to (V, P)$ is a $B$-birational isomorphism for $i=1, 2$, where $K_{S'_i}+\Delta_{S'_i}:=(K_{X'}+\Delta')|_{S'_i}$.

Suppose that $(S'_1, \Delta_{S'_1})$ is not kawamata log terminal. 
Then $(S'_2, \Delta_{S'_2})$ is also not kawamata log terminal. 
In this case, by applying Lemma \ref{p-lem4.2} to $\varphi\colon X'\to V$, we get a contradiction since $(X', \Delta')$ is kawamata log terminal outside $\lfloor \Delta'\rfloor$. 
This means that $(S'_1, \Delta_{S'_1})$ is kawamata log terminal. 
The same argument applies even if we interchange $S'_1$ and $S'_2$, which implies that $(S'_i, \Delta_{S'_i})$ is kawamata log terminal for $i=1, 2$. 
By inversion of adjunction, $(X', \Delta')$ is purely log terminal. 
Hence, $(X, \Delta)$ is purely log terminal. 
Clearly, $(S_1, \Delta_{S_1})$ is $B$-birationally equivalent to $(S_2, \Delta_{S_2})$. 
This completes the proof. 
\end{proof}

The following lemma is crucial.

\begin{lem}\label{p-lem4.5}
Let $(X', \Delta')$ be a 
log canonical pair and let $\pi'\colon X'\to Y$ be a projective 
surjective morphism of normal complex varieties. 
Let $W$ be a Stein compact subset of $Y$ such 
that $\Gamma (W, \mathcal O_Y)$ is noetherian. 
Assume that $X'$ is $\mathbb Q$-factorial 
over $W$. 
Let $f'\colon X'\to Z$ be 
a projective surjective morphism of normal 
complex varieties over $Y$ such that 
$K_{X'}+\Delta'\sim _{\mathbb Q, f'} 0$, and 
$\pi_Z\colon Z\to Y$ is projective, where 
$\pi_Z$ is the structure morphism. 
Assume that $(X', \Delta'-\varepsilon \lfloor \Delta'\rfloor)$ is 
kawamata log terminal for some small positive rational number $\varepsilon$ and 
there exists a $(K_{X'}+\Delta'-\varepsilon \lfloor\Delta'\rfloor)$-negative 
extremal Fano contraction $\varphi:=\varphi_R\colon X'\to V$ over $Z$ associated 
to an extremal ray $R$ of $\overline {\NE}(X'/Z; \pi_Z^{-1}(W))$ with 
$\dim V=\dim X'-1$. 
Here, we note that $V$ is $\mathbb Q$-factorial over $W$ and has only kawamata log terminal singularities after shrinking $Y$ around $W$ suitably.
\begin{equation}
\xymatrix{
X'\ar[dr]^-{f'}\ar[d]_-{\pi'}\ar[r]^-\varphi& V\ar[d] \\ 
Y & Z\ar[l]^-{\pi_Z}
}
\end{equation}
Then the horizontal part $(\Delta')^h$ of $\lfloor \Delta'\rfloor$ 
with respect 
to $\varphi$ satisfies one of the following conditions. 
\begin{itemize}
\item[{\em{(I)}}] $(\Delta')^h=D'_1$, which is irreducible, 
and $\deg [D'_1: V]=1$. 
\item[{\em{(II)}}] $(\Delta')^h=D'_1+D'_2$ such that $D'_i$ is irreducible 
and $\deg[D'_i: V]=1$ for 
$i=1, 2$. 
\item[{\em{(III)}}] $(\Delta')^h=D'_1$, which is irreducible,  
and $\deg [D'_1: V]=2$. 
\end{itemize}
We define $\Delta_{D'_i}$ by 
\begin{equation}
K_{D'_i}+\Delta_{D'_i}=(K_{X'}+\Delta')|_{D'_i}
\end{equation} 
for $i=1, 2$. Let $\nu_i\colon D'^\nu_i\to D'_i$ be 
the normalization for $i=1, 2$. 
We put 
\begin{equation}
K_{D'^\nu_i}+\Delta_{D'^\nu_i} :=\nu^*_i 
(K_{D'_i}+\Delta_{D'_i})
\end{equation} 
for $i=1, 2$. 
After shrinking $Y$ around $W$ suitably, we have the following 
statements. 
\begin{case}[I]
$\lfloor \Delta'\rfloor\cap \varphi^{-1}(v)$ is connected 
for every $v\in V$. 
\end{case}
\begin{case}[II]
The number of the connected components of $\lfloor \Delta'\rfloor \cap 
\varphi^{-1}(v)$ is at most two for every $v\in V$ and 
\begin{equation}\label{p-eq4.1}
\left(\varphi\circ \nu_2\right)^{-1}\circ 
\left(\varphi\circ \nu_1\right)\colon 
(D'^\nu_1, \Delta_{D'^\nu_1})\dashrightarrow 
(D'^\nu_2, \Delta_{D'^\nu_2})
\end{equation} 
is a $B$-bimeromorphic map over $V$. 
\end{case}
\begin{case}[III]
The number of the connected components of 
$\lfloor \Delta'\rfloor \cap \varphi^{-1}(v)$ 
is at most two for every $v\in V$ and 
there exists a $B$-bimeromorphic 
map  
\begin{equation}\label{p-eq4.2}
\iota \colon 
(D'^\nu_1, \Delta_{D'^\nu_1})\dashrightarrow 
(D'^\nu_1, \Delta_{D'^\nu_1})
\end{equation}
over $V$ with $\iota\ne \id$ and $\iota^2=\id$. 
\end{case}
Moreover, in {\em{(II)}} and {\em{(III)}}, if $\lfloor 
\Delta'\rfloor\cap\varphi^{-1}(v)$ is not connected 
for some $v\in V$, then $(X', \Delta')$ is purely 
log terminal in a neighborhood of $\varphi^{-1}(v)$. 
\end{lem}

More details on Cases (II) and (III) will be discussed in the following 
proof. 

\begin{proof}[Proof of Lemma \ref{p-lem4.5}] 
Since the proof is rather long, we divide it into three steps. 
\setcounter{step}{0}
\begin{step}\label{p-lem4.5-step1}
Since $\varphi\colon X'\to V$ is a $(K_{X'}+\Delta'-\varepsilon \lfloor\Delta'\rfloor)$-negative extremal Fano contraction, $V$ is $\mathbb Q$-factorial over $W$ by Lemmas \ref{f-lem7.3} and \ref{f-lem7.4} 
(see Section \ref{f-sec7}, which is an appendix). In particular, $K_V$ is $\mathbb Q$-Cartier after shrinking $Y$ around $W$ suitably. Then, by Corollary \ref{f-cor7.6}, $V$ has only kawamata log terminal singularities.

We have $R^i\varphi_*\mathcal O_{X'}=0$ for $i>0$ by the relative 
Kawamata--Viehweg vanishing theorem. Therefore, we see that 
general fibers of $\varphi\colon X'\to V$ are isomorphic to $\mathbb P^1$. 
Hence, the degree of $(\Delta')^h \to V$, where 
$(\Delta')^h$ is the horizontal part of $\lfloor \Delta'\rfloor$, is at most two. Therefore, we have (I), (II), and (III). 
We note that $\lfloor \Delta'\rfloor$ has normal crossings in codimension one since $(X', \Delta')$ is log canonical. 
\end{step}

\begin{step}\label{p-lem4.5-step2} 
Throughout this step, we will freely shrink $Y$ around $W$ without explicit mention. Hence, we may assume that $Y$ is Stein and $V$ is projective over $Y$. We note that the vertical part of $\lfloor \Delta'\rfloor$ is the pull-back of some effective $\mathbb Q$-divisor on $V$ since $\varphi\colon X'\to V$ is a $(K_{X'}+\Delta' -\varepsilon \lfloor \Delta'\rfloor)$-negative extremal contraction.

In Case (I), $(\Delta')^h=D'_1$ is irreducible 
and $\varphi$-ample. 
Since $\deg [D'_1: V]=1$, it follows from the Stein factorization and Zariski's main theorem that the intersection $D'_1\cap \varphi^{-1}(v)$ is connected for every $v\in V$. 
Since the vertical 
part of $\lfloor \Delta'\rfloor$ is the pull-back 
of some 
effective $\mathbb Q$-divisor on $V$, $\lfloor \Delta'\rfloor\cap \varphi^{-1}(v)$ is connected 
for every $v\in V$. 

In Case (II), we consider the following commutative diagram 
\begin{equation}
\xymatrix{
D'_i\ar[d] & D'^\nu_i\ar[dl]^{\rho_i}\ar[l]_-{\nu_i}\\ 
\ar[d]_-{\psi_i}D^\dag_i & \\ 
V &
}
\end{equation}
where $D'_i\to D^\dag_i\to V$ is the Stein factorization for $i=1, 2$. 
Since the mapping degree 
$\deg [D'_i: V]=1$, 
$\psi_i\colon D^\dag_i\to V$ is an isomorphism 
for $i=1, 2$. 
We put 
\begin{equation}
K_{D^\dag_i}+\Delta_{D^\dag_i}:=\rho_{i*}(K_{D'^\nu_i}+\Delta_{D'^\nu_i})
\end{equation} 
for $i=1, 2$. 
Let $\pi_V\colon V\to Y$ be the structure morphism. 
By taking $(\dim V-1)$ general hyperplane cuts 
with respect to a $\pi_V$-very ample Cartier divisor and 
applying \cite[12.3.4 Theorem]{flips-and-abundance}, 
we can find an effective $\mathbb Q$-divisor 
$P$ on $V$ such that 
$\psi_i\colon (D^\dag_i, \Delta_{D^\dag_i})\to (V, P)$ 
is a $B$-bimeromorphic isomorphism for $i=1, 2$. 
Consequently, 
\begin{equation}
\psi^{-1}_2\circ \psi_1\colon 
(D^\dag_1, \Delta_{D^\dag_1})\to 
(D^\dag_2, \Delta_{D^\dag_2})
\end{equation} 
is a $B$-bimeromorphic isomorphism. 
By construction, we have 
\[
K_{D'^\nu_i}+\Delta_{D'^\nu_i}=\rho^*_i(K_{D^\dag_i}+\Delta_{D^\dag_i})
\] 
for $i=1, 2$. 
Thus, 
\begin{equation}
\left(\varphi\circ \nu_2\right)^{-1}\circ 
\left(\varphi\circ \nu_1\right)
=\rho^{-1}_2\circ \psi^{-1}_2\circ \psi_1\circ \rho_1\colon 
(D'^\nu_1, \Delta_{D'^\nu_1})\dashrightarrow 
(D'^\nu_2, \Delta_{D'^\nu_2})
\end{equation} 
is a $B$-bimeromorphic map over $V$.

In Case (III), we 
consider the following commutative diagram 
\begin{equation}
\xymatrix{
D'_1\ar[d] & D'^\nu_1
\ar[d]^{\rho_1}\ar[l]_-{\nu_1}\\ 
D^\dag_1\ar[d] 
& D_1^{\dag \nu}\ar[l]_-{\nu^\dag_1}\ar[dl]^-p\\ 
V &
}
\end{equation} 
where $D'_1\to D^\dag_1\to V$ is the Stein factorization and 
$\nu^\dag_1\colon D^{\dag\nu}_1\to D^\dag_1$ is the normalization. 
Let $p\colon D^{\dag\nu}_1\to V$ be the induced morphism. 
Note that $p\colon D^{\dag\nu}_1\to V$ is a finite morphism 
of normal complex varieties of degree two. 
Hence, there exists a nonempty Zariski open subset $V^\circ$ of $V$ such that 
$p\colon p^{-1}(V^\circ)\to V^\circ$ is a finite morphism of complex manifolds 
of degree two. 
Therefore, there exists a Galois involution $\iota^\circ \colon p^{-1}(V^\circ)
\to p^{-1}(V^\circ)$ over $V^\circ$. 
By \cite[Lemma 2.24]{enokizono-hashizume-ss}, 
$\iota^\circ$ extends to a birational involution $\iota^\dag\colon D^{\dag\nu}_1\to D^{\dag\nu}_1$ over $V$ such that 
$\iota^\dag\ne \id$ and $(\iota^\dag)^2=\id$. 
We put 
\begin{equation} 
K_{D^{\dag\nu}_1}+\Delta_{D^{\dag\nu}_1}:=
\rho_{1*}(K_{D'^\nu_1}+\Delta_{D'^\nu_1}). 
\end{equation} 
Over a nonempty Euclidean open subset of $V$ where $D'_1$ is a 
union of two sections, the situation is the same as in Case (II). 
On the other hand, over the locus where $D'_1\to V$ is a ramified double cover of smooth varieties, 
both $D'_1\to D^\dag_1$ and $D^{\dag\nu}_1\to D^\dag_1$ are 
isomorphisms, and the ramification locus of $D^{\dag\nu}_1\to V$ is $\iota^\dag$-invariant. 
Hence, we see that 
$\iota^\dag$ preserves $\Delta_{D^{\dag\nu}_1}$. 
Thus, $\iota^\dag$ preserves $K_{D^{\dag\nu}_1}+\Delta_{D^{\dag\nu}_1}$. 
Therefore, since 
\[
K_{D'^\nu_1}+\Delta_{D'^\nu_1}=\rho_1^*(
K_{D^{\dag\nu}_1}+\Delta_{D^{\dag\nu}_1}), 
\]
we obtain a $B$-bimeromorphic 
involution map 
\begin{equation}
\iota :=\rho_1^{-1}\circ \iota^\dag\circ \rho_1\colon (D'^{\nu}_1, \Delta_{D'^{\nu}_1})
\dashrightarrow 
(D'^{\nu}_1, \Delta_{D'^{\nu}_1})
\end{equation} 
over $V$. 
\end{step}

\begin{step}\label{p-lem4.5-step3}
Assume that $\lfloor \Delta'\rfloor\cap \varphi^{-1}(v)$ is not connected 
in Cases (II) and (III). Since the vertical part of 
$\lfloor \Delta'\rfloor$ is the pull-back of some effective 
$\mathbb Q$-divisor on $V$, 
by shrinking $V$ around $v$, we may assume that $\lfloor \Delta'\rfloor$ has no vertical components. From now on, we freely shrink $V$ around $v$ without explicit mention. 

We first treat Case (II). If $(D'^\nu_1, \Delta_{D'^\nu_1})$ is 
not kawamata log terminal over $v$, then $(D'^\nu_2, \Delta_{D'^\nu_2})$ 
is not kawamata log terminal over $v$ by \eqref{p-eq4.1}. 
In this case, there exist a log canonical center strictly contained in $D'_1$ and another log canonical center strictly contained in $D'_2$. 
By Lemma \ref{p-lem4.2}, these two log canonical centers must be connected by a chain of log canonical centers. 
This is a contradiction, since $(X', \Delta')$ is kawamata log terminal outside $\lfloor \Delta'\rfloor$. The same argument applies even if we interchange $D'_1$ and $D'_2$.

Next, we treat Case (III). We may assume that $D'_1$ has exactly two connected components by shrinking $V$ around $v$. 
If one of the connected components of $(D'^\nu_1, \Delta_{D'^\nu_1})$ is not kawamata log terminal over $v$, then the other connected component is also not kawamata log terminal over $v$ by \eqref{p-eq4.2}. 
Then, applying Lemma \ref{p-lem4.2} just as in Case (II) yields a contradiction. 

Therefore, we see that $(D'^\nu_i, \Delta_{D'^\nu_i})$ is kawamata log terminal over $v$ for $i=1, 2$. 
By inversion of adjunction, this means that $(X', \Delta')$ is purely log 
terminal in a neighborhood of $\varphi^{-1}(v)$ since 
$(X', \Delta')$ is kawamata log terminal outside $\lfloor \Delta'\rfloor$.
\end{step}

We finish the proof of Lemma \ref{p-lem4.5}. 
\end{proof}

By Lemma \ref{p-lem4.5}, we have: 

\begin{prop}\label{p-prop4.6}
Let $(X, \Delta)$ be a 
divisorial log terminal pair and let $\pi\colon X\to Y$ be a projective 
surjective morphism of normal complex varieties. 
Let $W$ be a Stein compact subset of $Y$ such 
that $\Gamma (W, \mathcal O_Y)$ is noetherian. 
Assume that $X$ is $\mathbb Q$-factorial over $W$. 
Let $f\colon X\to Z$ be a projective surjective morphism of normal 
complex varieties over $Y$ such that 
$f_*\mathcal O_X\simeq \mathcal O_Z$, 
$K_X+\Delta\sim _{\mathbb Q, f} 0$, and 
$\pi_Z\colon Z\to Y$ is projective, where 
$\pi_Z$ is the structure morphism. 
Then, after shrinking $Y$ around $W$ suitably, 
the number of the connected components of $\lfloor \Delta
\rfloor \cap f^{-1}(z)$ is at most two for every $z\in Z$. 
If $\lfloor \Delta\rfloor \cap f^{-1}(z)$ is not connected 
for some $z\in Z$,  
then there exists a meromorphic map 
$q \colon X\dashrightarrow V$ over $Z$ whose general fiber 
is $\mathbb P^1$ such that 
$V$ is $\mathbb Q$-factorial over $W$ and has 
only kawamata log terminal singularities. The horizontal 
part $\Delta^{h}$ of $\lfloor \Delta\rfloor$ with respect 
to $q$ satisfies one of the following conditions. 
\begin{itemize}
\item[{\em{(i)}}] $\Delta^h=D_1$, which 
is irreducible, the mapping degree $\deg [D_1: V]=2$, 
and there is a $B$-bimermorphic 
involution on $(D_1, \Delta_{D_1})$ 
over $Z$. 
\item[{\em{(ii)}}] 
$\Delta^h=D_1+D_2$ such that 
$D_i$ is irreducible for $i=1, 2$ and 
\begin{equation}
(q|_{D_2})^{-1}\circ (q|_{D_1})\colon 
(D_1, \Delta_{D_1})\dashrightarrow (D_2, \Delta_{D_2}) 
\end{equation} 
is a $B$-bimeromorphic map over $Z$. 
\end{itemize} 
We note that 
$K_{D_i}+\Delta_{D_i}:=(K_X+\Delta)|_{D_i}$ and 
$(D_i, \Delta_{D_i})$ is divisorial log 
terminal for $i=1, 2$. 
More precisely, 
by a $(K_X+\Delta-\varepsilon \lfloor 
\Delta\rfloor)$-minimal model program with ample scaling 
over $Z$ around $\pi^{-1}_Z(W)$, after shrinking 
$Y$ around $W$ suitably, we have $p\colon (X, \Delta)\dashrightarrow 
(X', \Delta')$ over $Z$ and $(X', \Delta')$ satisfies {\em{(II)}} or 
{\em{(III)}} in Lemma \ref{p-lem4.5}. 
\begin{equation}
\xymatrix{
X \ar@{-->}[rrd]_-q
\ar[ddr]_-f\ar[dd]_\pi\ar@{-->}[rr]^-p& &X'\ar[d]^-\varphi\\
&& V\ar[dl]\\  
Y & Z\ar[l]^-{\pi_Z} & 
}
\end{equation}
\end{prop}

The reader can find more details in the following proof. 

\begin{proof}[Proof of Proposition \ref{p-prop4.6}] 
The idea of the proof is very simple. 
By running a suitable minimal model program, 
we reduce the problem to Lemma \ref{p-lem4.5}. 
We note that we need the minimal model program established 
in \cite{enokizono-hashizume-mmp} in Step \ref{p-prop4.6-step1}. 
The minimal model program treated in \cite{fujino-bchm} 
is sufficient for Step \ref{p-prop4.6-step2}. 
\setcounter{step}{0}
\begin{step}\label{p-prop4.6-step1}
In this step, we assume that $\lfloor \Delta\rfloor$ is 
not dominant onto $Z$. 
Under this assumption, we will prove that 
$\lfloor \Delta\rfloor\cap f^{-1}(z)$ is connected 
for every $z\in Z$. 

We take an arbitrary point $z\in Z$ and 
a Stein compact subset $W_z$ of $Z$ containing $z$ such that 
$\Gamma (W_z, \mathcal O_Z)$ is noetherian. 
By \cite[Theorem 1.2]{enokizono-hashizume-mmp}, 
we can run a $(K_X+\Delta-\varepsilon \lfloor 
\Delta\rfloor)$-minimal model 
program with ample scaling 
over $Z$ around $W_z$. 
We finally get a commutative diagram 
\begin{equation}
\xymatrix{
X \ar@{-->}[rr]^-p\ar[dr]_-f&&X'\ar[dl]^-{f'} \\ 
&Z&
}
\end{equation}
around $W_z$ such 
that $f'\colon X'\to Z$ is the structure morphism, 
$p$ is a finite composite of 
flips and divisorial contractions, and 
$K_{X'}+\Delta'-\varepsilon \lfloor \Delta'\rfloor$ is 
nef over $W_z$, where $\Delta':=p_*\Delta$. 
Thus, $-\lfloor \Delta'\rfloor$ is $f'$-nef 
since $K_{X'}+\Delta'\sim _{\mathbb Q, f'}0$. 
Suppose that $\lfloor \Delta'\rfloor \cap f'^{-1}(z)\ne \emptyset$ and 
$f'^{-1}(z)\not\subset \Supp \lfloor 
\Delta'\rfloor$. 
Then we can find a projective integral curve $C$ contained in 
$f'^{-1}(z)$ such that 
$C\not\subset \Supp \lfloor \Delta'\rfloor$ and 
$C\cap \Supp \lfloor \Delta'\rfloor\ne \emptyset$. 
In this case, we have $-\lfloor \Delta'\rfloor \cdot C<0$. 
This is a contradiction. 
Hence, we have $\lfloor \Delta'\rfloor \cap f'^{-1}(z)=\emptyset$ or 
$f'^{-1}(z)\subset \Supp \lfloor \Delta'\rfloor$. 
In particular, $\lfloor \Delta'\rfloor \cap f'^{-1}(z)$ 
is connected. 
Since the number of connected components 
of $\lfloor \Delta\rfloor \cap f^{-1}(z)$ 
is preserved under the above minimal model 
program by Lemma \ref{p-lem4.1}, 
we conclude that $\lfloor \Delta\rfloor \cap f^{-1}(z)$ is connected. 
\end{step}

\begin{step}\label{p-prop4.6-step2}
In this step, we assume that $\lfloor \Delta\rfloor$ is dominant onto $Z$. 
Then $K_X+\Delta-\varepsilon \lfloor\Delta\rfloor$ is not pseudo-effective over $Z$. 
By \cite[Theorem 1.1 and Lemma 9.4]{fujino-bchm}, we can run a $(K_X+\Delta-\varepsilon\lfloor \Delta\rfloor)$-minimal model program with ample scaling over $Z$ around $W_Z:=\pi^{-1}_Z(W)$. 
Then we obtain a finite sequence of divisorial contractions and flips 
\begin{equation}
p\colon X=:X_0\dashrightarrow X_1\dashrightarrow \cdots \dashrightarrow X_m=:X'
\end{equation} 
such that there exists a $(K_{X'}+\Delta'-\varepsilon \lfloor \Delta'\rfloor)$-negative extremal Fano contraction $\varphi\colon X'\to V$ over $Z$. 

\begin{claim}\label{p-prop4.6-claim}
If $\dim V\leq \dim X-2$, then $\lfloor \Delta'\rfloor \cap \varphi^{-1}(v)$ is connected for every $v\in V$. 
\end{claim} 

\begin{proof}[Proof of Claim] 
Since the vertical part of $\lfloor \Delta'\rfloor$ is the pull-back of some effective $\mathbb Q$-divisor on $V$, it suffices to prove that $(\Delta')^h\cap \varphi^{-1}(v)$ is connected for every $v\in V$, where $(\Delta')^h$ is the horizontal part of $\lfloor \Delta'\rfloor$. 
Since $(\Delta')^h$ is $\varphi$-ample, $(\Delta')^h\cap \varphi^{-1}(v)$ is connected for general $v\in V$. 
By the Stein factorization and Zariski's main theorem, we see that $(\Delta')^h\cap \varphi^{-1}(v)$ is connected for every $v\in V$. This completes the proof of the claim. 
\end{proof}
The claim implies that $\lfloor \Delta'\rfloor \cap (f')^{-1}(z)$ is connected for every $z\in Z$ when $\dim V\leq \dim X-2$. 
Since the above minimal model program preserves the number of connected components of $\lfloor \Delta\rfloor \cap f^{-1}(z)$ by Lemma \ref{p-lem4.1}, $\lfloor \Delta\rfloor \cap f^{-1}(z)$ is connected for every $z\in Z$. 

Hence, from now on, we may assume that $\dim V=\dim X-1$. 
In this case, we have already described the situation in Lemma \ref{p-lem4.5}. 
Case (III) (resp.~(II)) in Lemma \ref{p-lem4.5} implies (i) (resp.~(ii)). 
\end{step} 
We finish the proof of Proposition \ref{p-prop4.6}. 
\end{proof}

Before presenting 
our gluing argument, 
we state an elementary but important lemma. 
Here, we need the finiteness of relative log pluricanonical 
representations (see Corollary \ref{b-cor1.3}). 

\begin{lem}[{\cite[Lemma 4.9]{fujino-abundance}}]\label{p-lem4.7}
Let $(X, \Delta)$ be an equidimensional {\em{(}}not necessarily 
connected{\em{)}} divisorial 
log terminal pair and let $\pi\colon X\to Y$ be 
a projective 
morphism of complex analytic spaces such that 
$K_X+\Delta$ is $\pi$-semiample. 
Let $W$ be a compact subset of $Y$ 
and let 
$U$ be a Stein open subset of $Y$ with $U\subset W$. 
We further assume that $U$ is semianalytic. 
We put $G:=\rho^{WU}_m(\Bim(X/Y, \Delta; W))$. 
Then $G$ is a finite group. 
We put $X_U:=\pi^{-1}(U)$. 
If 
\begin{equation}
s\in \PA\left(X_U, \mathcal O_X
(m(K_X+\Delta))\right),
\end{equation} 
then 
$g^*s|_{\lfloor \Delta\rfloor}=s|_{\lfloor \Delta\rfloor}$ and 
\begin{equation}
g^*s \in \PA\left(X_U, \mathcal O_X
(m(K_X+\Delta))\right)
\end{equation} 
for 
every $g\in G$. 
In particular, 
\begin{equation}
\sum_{g\in G}g^*s\in A\left(X_U, \mathcal O_X
(m(K_X+\Delta))\right), 
\end{equation}
\begin{equation}
\prod_{g\in G}g^*s\in A\left(X_U, \mathcal O_X
(m|G|(K_X+\Delta))\right), 
\end{equation} 
and 
\begin{equation}
\prod _{g\in G}g^*s|_{\lfloor \Delta\rfloor} 
=\left(s|_{\lfloor 
\Delta\rfloor}\right)^{|G|}. 
\end{equation} 
Of course, 
\begin{equation}
\frac{1}{|G|} \sum _{g\in G} g^*s|_{\lfloor \Delta\rfloor}=
s|_{\lfloor \Delta\rfloor}
\end{equation} 
holds. 
\end{lem}
\begin{proof}
By Corollary \ref{b-cor1.3}, 
$G$ is a finite group. 
Then, by Lemma \ref{c-lem2.30}, it is not difficult to see 
that the proof of \cite[Lemma 4.9]{fujino-abundance} works 
in our complex analytic setting. 
Hence we omit the details here. 
\end{proof}

The following lemma is required for our inductive gluing argument.

\begin{lem}[{\cite[Lemma 4.7]{fujino-abundance}}]\label{p-lem4.8} 
Let $\pi\colon X\to Y$ be a projective morphism 
of complex analytic spaces such that 
$(X, \Delta)$ is an equidimensional divisorial log terminal 
pair and that $K_X+\Delta$ is 
$\pi$-semiample. Let $W$ be a compact subset of $Y$ and 
let $U$ be a semianalytic Stein open subset 
of $Y$ with $U\subset W$. 
Assume that $\PA(X_U, \mathcal O_X(m(K_X+\Delta)))$ 
generates $\mathcal O_X(m(K_X+\Delta))$ over $U$. 
Then there exists a sufficiently divisible positive integer $m'$ 
such that 
$\A(X_U, \mathcal O_X(m'm(K_X+\Delta)))$ generates 
$\mathcal O_X(m'm(K_X+\Delta))$ 
over $U$. 
\end{lem}

\begin{proof} 
We put $G:=\rho^{WU}_m\left(\Bim(X/Y, \Delta; W)\right)$. 
Then $G$ is a finite group by Corollary \ref{b-cor1.3}. 
We put $G:=\{g_1, \cdots, g_N\}$ with 
$N:=|G|$. 
Let $\sigma_i$ be the $i$th 
elementary symmetric polynomial for $1\leq i\leq N$. 
Then we have 
\begin{equation}
\{s=0\}\supset \bigcap _{j=1}^N \{g^*_j s=0\} =\bigcap _{i=1}^N 
\{\sigma_i(g^*_1s, \cdots, g^*_Ns)=0\}. 
\end{equation} 
By Lemma \ref{p-lem4.7}, we see that 
\[
\sigma_i(g^*_1s, \cdots, g^*_Ns)\in A(X_U, \mathcal O_X(im(K_X+\Delta)))
\] 
for every $i$. 
Therefore, by considering 
\begin{equation}
\sigma^{N!/i}_i(g^*_1s, \cdots, g^*_Ns)\in 
\A(X_U, \mathcal O_X(N!m(K_X+\Delta)))
\end{equation} 
for $s\in \PA(X_U, \mathcal O_X(m(K_X+\Delta)))$, 
we can check that $\A(X_U, \mathcal O_X(N!m(K_X+\Delta)))$ 
generates $\mathcal O_X(N!m(K_X+\Delta))$ over $U$ under the 
assumption that 
$\PA(X_U, \mathcal O_X(m(K_X+\Delta)))$ generates 
$\mathcal O_X(m(K_X+\Delta))$ over $U$. 
Thus we obtain the desired statement of Lemma \ref{p-lem4.8}. 
We finish the proof. 
\end{proof}

Proposition \ref{p-prop4.9}, which is 
essentially the same as \cite[Proposition 4.5]{fujino-abundance}, 
is a key step of our 
gluing argument. 

\begin{prop}[{\cite[Proposition 4.5]{fujino-abundance}}]\label{p-prop4.9}
Let $\pi\colon X\to Y$ be a projective 
morphism of complex analytic spaces such that $(X, \Delta)$ is divisorial 
log terminal. 
Let $U$ be a semianalytic Stein open subset of $Y$ 
and let $W$ be a Stein compact subset of $Y$ with 
$U\subset W$ such that $X$ is $\mathbb Q$-factorial over $W$ 
and that $\Gamma (W, \mathcal O_Y)$ is noetherian. 
We put $S:=\lfloor \Delta\rfloor$, $X_U:=\pi^{-1}(U)$, and 
$S_U:=S|_{\pi^{-1}(U)}$. 
Assume that 
\begin{itemize}
\item[{\em{(1)}}] 
$K_X+\Delta$ is $\pi$-semiample, and 
\item[{\em{(2)}}] 
$\A\left (S_U, 
\mathcal O_S(m_0(K_X+\Delta))\right)$ generates 
$\mathcal O_S(m_0(K_X+\Delta))$ over $U$ for some positive 
integer $m_0$. 
\end{itemize}
If necessary, we replace $U$ with a smaller 
semianalytic Stein open subset of $Y$. 
Then there exists a positive integer $m_1$ such that 
$m_1m_0\in 2\mathbb Z$, the natural restriction map 
\begin{equation}
\PA \left(X_U, \mathcal O_X(m_1m_0(K_X+\Delta))\right)\to 
\A\left (S_U, 
\mathcal O_S(m_1m_0(K_X+\Delta))\right)
\end{equation} 
is surjective, and 
$\PA \left(X_U, \mathcal O_X
(m_1m_0(K_X+\Delta))\right)$ generates $\mathcal O_X(m_1m_0(K_X+\Delta))$ 
over $U$. 
\end{prop}

Since we are working with complex analytic spaces, 
there are additional technical difficulties. 
For example, we need to take a Stein compact subset $W$ and a semianalytic 
Stein open subset $U \subset W$, and shrink $U$ further if necessary. 
Nevertheless, the following proof is essentially 
the same as that of \cite[Proposition 4.5]{fujino-abundance}. 
We describe it here for the reader's convenience.

\begin{proof}[Proof of Proposition \ref{p-prop4.9}]
It suffices to prove this proposition for each connected component of $X$. 
Hence we may assume that $X$ is irreducible. 
Throughout this proof, we may freely shrink $Y$ around $W$ without 
explicit mention. 
We first take a relative Iitaka fibration $f\colon X\to Z$ over $Y$, 
that is, $f\colon X\to Z$ is a projective surjective morphism 
of normal complex analytic varieties such that 
$f_*\mathcal O_X\simeq \mathcal O_Z$ and 
$\mathcal O_X(m(K_X+\Delta))\simeq f^*\mathcal L$ holds for 
some positive integer $m$ and a $\pi_Z$-ample 
line bundle $\mathcal L$ on $Z$, 
where $\pi_Z\colon Z\to Y$ is the structure morphism:
\[
\xymatrix{
X \ar[dr]_-\pi\ar[rr]^-f&& Z\ar[dl]^-{\pi_Z}\\ 
&Y&
}
\]
If $S=\lfloor \Delta\rfloor=0$, then there is nothing to prove. 
Therefore, we may assume that $S=\lfloor \Delta\rfloor\neq 0$. 
Then we have the following four cases: 
\begin{itemize}
\item[(1)] $Z$ is a point and $S$ is connected,
\item[(2)] $\dim Z\geq 1$, $S\cap f^{-1}(z)$ is 
connected for every $z\in Z$, and $f(S)=Z$,
\item[(3)] $\dim Z\geq 1$, $S\cap f^{-1}(z)$ is 
connected for every $z\in Z$, and $f(S)\subsetneq Z$, and
\item[(4)] $S \cap f^{-1}(z)$ is not 
connected for some $z\in Z$. 
\end{itemize}

\setcounter{step}{0}
\begin{step}\label{p-prop4.9-step1}
In this step, we will treat (1). 

When $Z$ is a point, $X$ is projective and $K_X+\Delta\sim _{\mathbb Q}0$. 
We consider the following long exact sequence: 
\begin{equation}
\begin{split}
0&\to H^0(X, \mathcal O_X(m_0(K_X+\Delta)-S))\to 
H^0(X, \mathcal O_X(m_0(K_X+\Delta)))\\ &\to 
H^0(S, \mathcal O_S(m_0(K_X+\Delta)))
\to \cdots. 
\end{split}
\end{equation}
Since $K_X+\Delta\sim _{\mathbb Q}0$ and $S \ne 0$, 
we obtain that $H^0(X, \mathcal O_X(m_0(K_X+\Delta)-S))=0$ and 
that the second and the third terms are one-dimensional. 
Hence we obtain the desired statement. 
\end{step}
\begin{step}\label{p-prop4.9-step2}
In this step, we will treat (2). 

By taking a divisible positive integer $m$ such that 
$\A(S_U, \mathcal O_S(m(K_X+\Delta)))$ generates $\mathcal O_S(m(K_X+\Delta))$ 
over $U$ and that $\mathcal O_X(m(K_X+\Delta))\simeq f^*\mathcal L$ holds 
for some $\pi_Z$-ample line bundle $\mathcal L$ on $Z$. 
If necessary, we replace $U$ with a smaller relatively 
compact semianalytic Stein open subset of $Y$. 
By $\A\left (S_U, 
\mathcal O_S
(m(K_X+\Delta))\right)$, we can construct a morphism $\Phi\colon 
S \to Z'$ over $U$. 
Since every curve in any fiber of $f|_S$ over 
$U$ is mapped to a point by $\Phi$, 
there exists a morphism $\Psi\colon Z\to Z'$ over $U$ such that 
$\Psi\circ (f|_S)=\Phi$. 
Over $U$, there exists the following commutative diagram. 
\begin{equation}
\xymatrix{
X \ar[d]_-f& S \ar@{_{(}->}[l]\ar[d]^-\Phi\\ 
Z \ar[r]_-\Psi & Z' 
}
\end{equation}
We note that 
\begin{equation}
\xymatrix{
\Phi\colon S \ar[r]^-{f|_S}& Z
\ar[r]^-\Psi&Z'
}
\end{equation} 
and that $f|_S$ is surjective with connected fibers. 
For any 
\begin{equation}
s\in \A\left (S_U, 
\mathcal O_S
(m(K_X+\Delta))\right), 
\end{equation} 
we can take $t$ such that 
$s=\Phi^*t$. 
We put $u:=f^*\Psi^*t$. 
Then 
\begin{equation}
u\in \PA \left(X_U, \mathcal O_X
(m(K_X+\Delta))\right)
\end{equation} such that $u|_S=s$. 
By construction, 
$\PA \left(X_U, \mathcal O_X
(m(K_X+\Delta))\right)$ generates $\mathcal O_X(m(K_X+\Delta))$ 
over $U$. 
\end{step}

\begin{step}\label{p-prop4.9-step3}
In this step, we will treat (3). 

This step is a relative version of \cite[Lemma 4.3]{fujino-abundance}. 
We take a divisible positive integer $m$ such that 
$\mathcal O_X(m(K_X+\Delta))\simeq f^*\mathcal L$ for some 
$\pi_Z$-ample line bundle $\mathcal L$ on $Z$. 
\begin{equation} 
\xymatrix{
X\ar[dr]_-{\pi} \ar[rr]^-f&& Z\ar[dl]^-{\pi_Z} \\
&Y&
}
\end{equation}
We put $T:=f(S)\subsetneq Z$. 
Then $f_*\mathcal O_S\simeq \mathcal O_T$ by Corollary 
\ref{p-cor4.3}. Therefore, we have the following commutative 
diagram: 
\begin{equation}\label{p-eq4.1}
\xymatrix{
\pi_*\mathcal O_X(lm(K_X+\Delta))\ar[r] & \pi_*\mathcal O_S(lm(K_S+\Delta_S))\\ 
\pi_{Z*}\mathcal L^{\otimes l}\ar[r]\ar[u]^-{\simeq}& 
\pi_{Z*} \left(\mathcal L^{\otimes l}|_T\right)\ar[u]^-{\simeq}. 
}
\end{equation}
Note that the vertical arrows are isomorphisms. 
If we replace $U$ with a relatively compact semianalytic 
Stein open subset and make $l$ sufficiently large, 
then $\mathcal L^{\otimes l}\otimes \mathcal I_T$ is 
$\pi_Z$-generated over $U$, 
where $\mathcal I_T$ is the defining ideal sheaf 
of $T$ on $Z$, 
and $R^1\pi_{Z*}\left(\mathcal L^{\otimes l}\otimes \mathcal I_T\right)=0$ 
since $\mathcal L$ is $\pi_Z$-ample. 
Thus, by \eqref{p-eq4.1}, we have the following short exact sequence: 
\begin{equation}\label{p-eq4.2}
\begin{split}
0&\to \pi_*\mathcal O_X(lm(K_X+\Delta)-S)\to 
\pi_*\mathcal O_X(lm(K_X+\Delta))
\\ &\to 
\pi_*\mathcal O_S(lm(K_S+\Delta_S))\to 0. 
\end{split}
\end{equation}
By definition, it is obvious that 
every element of $H^0(X_U, \mathcal O_X(lm(K_X+\Delta)-S))$ 
is contained in $\PA(X_U, \mathcal O_X(lm(K_X+\Delta)))$. 
By \eqref{p-eq4.2}, we can extend 
\begin{equation}
\A\left (S_U, 
\mathcal O_S
(lm(K_S+\Delta_S))\right)
\end{equation} 
to 
\begin{equation}
\PA \left(X_U, \mathcal O_X(lm(K_X+\Delta))\right)
\end{equation}  
and check that 
$\PA \left(X_U, \mathcal O_X
(lm(K_X+\Delta))\right)$ generates $\mathcal O_X(lm(K_X+\Delta))$ 
over $U$. 
\end{step}

\begin{step}\label{p-prop4.9-step4} In this step, we will treat (4). 

Since $S\cap f^{-1}(z)$ is not connected for some $z\in Z$ by assumption, 
we can apply Proposition \ref{p-prop4.6}. 
More precisely, by Step \ref{p-prop4.6-step2} in the proof of Proposition \ref{p-prop4.6}, 
we run a $(K_X+\Delta-\varepsilon\lfloor\Delta\rfloor)$-minimal model program 
with ample scaling over $Z$ around $W_Z:=\pi_Z^{-1}(W)$ 
(see \cite[Theorem 1.2 and Lemma 9.4]{fujino-bchm}), and 
eventually obtain $(X', \Delta')$ together with 
a $(K_{X'}+\Delta'-\varepsilon \lfloor\Delta'\rfloor)$-negative extremal Fano contraction $\varphi\colon X'\to V$ 
as in Lemma \ref{p-lem4.5}. 
Then either (II) or (III) of Lemma \ref{p-lem4.5} holds. 
Henceforth, we freely use the notation introduced in Lemma \ref{p-lem4.5} and 
its proof. 
Note that $p\colon (X, \Delta)\dashrightarrow (X', \Delta')$ is $B$-bimeromorphic 
over $Y$. The situation is summarized in the following commutative diagram:
\begin{equation}
\xymatrix{
D_i\ar@{^{(}->}[d]\ar@{-->}[rr]& &D'_i\ar@{^{(}->}[d]&D'^\nu_i\ar[l]_-{\nu_i}
\ar[d]^-{\rho_i}\\
X \ar@{-->}[rrd]_-q
\ar[ddr]_-f\ar[dd]_\pi\ar@{-->}[rr]^-p& &X'\ar[d]_-\varphi&D^{\dag\nu}_i\ar[dl]
\\
&& V\ar[dl]&\\ 
Y & Z\ar[l]^-{\pi_Z} & &
}
\end{equation}
Let $m$ be a sufficiently large and divisible positive integer 
such that $\mathcal O_X(m(K_X+\Delta))\simeq f^*\mathcal L$ for 
some line bundle $\mathcal L$ on $Z$. 
We take any element $s\in \A(S_U, \mathcal O_S(m(K_S+\Delta_S)))$. 
By Remark \ref{c-rem2.21}, there exist natural isomorphisms 
\begin{equation}
H^0\left(X_U, \mathcal O_X(m(K_X+\Delta))\right)\simeq 
H^0\left(X'_U, \mathcal O_{X'}(m(K_{X'}+\Delta'))\right)
\end{equation} 
and 
\[
H^0\left(S_U, \mathcal O_S(m(K_S+\Delta_S))\right)\simeq 
H^0\left(S'_U, \mathcal O_{S'}(m(K_{S'}+\Delta_{S'}))\right)
\]
induced by $p$, where $\pi'\colon X'\to Y$, $X'_U:=\pi'^{-1}(U)$, 
$S':=\lfloor \Delta'\rfloor$, and $S'_U:=S'\cap X'_U$. 
Hence, $s$ induces 
\begin{equation}
s'\in H^0\left(S'_U, \mathcal O_{S'}(m(K_{S'}+\Delta_{S'}))\right). 
\end{equation} 
The section $s'$ induces a section 
\begin{equation}
s''_i\in H^0\left(D'^\nu_i, \mathcal O_{D'^\nu_i}(m(K_{D'^\nu_i}+\Delta_{D'^\nu_i}))\right)
\end{equation} 
over $U$ for $i=1, 2$. In Case (III), 
$s''_1$ is $\iota$-invariant. 
Hence, $s''_1$ descends to a section $t$ of 
$\mathcal L_V$ over $U$, where $\mathcal L_V$ is the 
pull-back of $\mathcal L$ to $V$. 
In Case (II), $s''_1$ also naturally descends to a section $t$ of 
$\mathcal L_V$ over $U$. 
In Case (III), the pull-back of $\varphi^*t$ 
to $D'^\nu_1$ coincides with $s''_1$ by construction. 
In Case (II), on a small Euclidean open subset $\widetilde U$ of $U$ such that 
$\varphi^{-1}(\widetilde U)\simeq \mathbb P^1\times \widetilde U$ and 
$\varphi|_{\varphi^{-1}(\widetilde U)}\colon \mathbb P^1\times \widetilde U \to \widetilde U$ is the second projection, 
we see that the pull-back of $\varphi^*t$ to $D'^\nu_2$ coincides with 
$(-1)^ms''_2$ over $\widetilde U$ by the local calculation 
in the proof of \cite[12.3.4 Theorem]{flips-and-abundance}. 
Therefore, the pull-back of $\varphi^*t$ to $D'^\nu_2$ coincides with $(-1)^ms''_2$. 
By construction, it is clear that the pull-back of 
$\varphi^*t$ to $D'^\nu_1$ coincides with $s''_1$. 
Hence, we have $s'|_{(\Delta')^h}=(\varphi^*t)|_{(\Delta')^h}$ 
if $m$ is even. 

Next, we show that 
$(\varphi^*t)|_{\lfloor \Delta'\rfloor}=s'|_{\lfloor \Delta'\rfloor}$ 
as in Case 4 in the proof of \cite[Proposition 4.5]{fujino-abundance}. 
Let $(\Delta')^v$ be the vertical part of $\lfloor \Delta'\rfloor$. 
We can write $(\Delta')^v=\sum_i \varphi^*P_i$ such that $Q_i:=\Supp P_i$ 
is a prime divisor on $V$ for every $i$ and $Q_i\neq Q_j$ for 
$i\neq j$. We put $E_i:=\varphi^*P_i$. 
Then it suffices to check that $s'|_{E_i}=(\varphi^*t)|_{E_i}$ 
for every $i$. 
Let $F_i$ be an irreducible component of $E_i\cap (\Delta')^h$ such that 
$\varphi\colon F_i\to Q_i$ is dominant. 
Since $(\Delta')^h\cap (\Delta')^v\neq \emptyset$, 
such an $F_i$ always exists. 
We consider the following commutative diagram: 
\begin{equation}
\xymatrix{
\pi_*\mathcal O_{E_i}(m(K_{X'}+\Delta')) \ar[r]& 
\pi_*\mathcal O_{F_i}(m(K_{X'}+\Delta')) \\ 
\pi_{V*}\left(\mathcal L_V|_{Q_i}\right) \ar[u]^-\simeq 
\ar@{=}[r]& 
\pi_{V*}\left(\mathcal L_V|_{Q_i}\right)\ar@{^{(}->}[u]_-{j}, 
}
\end{equation} 
where $\pi_V\colon V\to Y$ is the structure morphism. 
The left vertical arrow is an isomorphism by Lemma \ref{p-lem4.2}. 
The map $j$ is injective since $\varphi\colon F_i\to Q_i$ is dominant. 
Since $s'|_{F_i}=(\varphi^*t)|_{F_i}$, 
we have $s'|_{E_i}=(\varphi^*t)|_{E_i}$ for every $i$. 
Thus, we obtain $s'|_{\lfloor \Delta'\rfloor}=(\varphi^*t)|_{\lfloor \Delta'\rfloor}$. 
This means that $s$ can be lifted to a member of 
$\PA\left(X_U, \mathcal O_X(m(K_X+\Delta))\right)$. 
By construction, it is not difficult to see 
that $\PA\left(X_U, \mathcal O_X(m(K_X+\Delta))\right)$ generates $\mathcal O_X(m(K_X+\Delta))$ 
over $U$.
\end{step}
We finish the proof. 
\end{proof}

As a consequence of Lemma \ref{p-lem4.8} and Proposition \ref{p-prop4.9}, 
we obtain the following key lemma, which 
plays a crucial role in the proof of Theorem \ref{b-thm1.1}.

\begin{lem}[Abundance for semi-divisorial log terminal pairs 
in the complex analytic setting]\label{p-lem4.10}
Let $(X, \Delta)$ be a semi-divisorial log terminal pair and 
let $\pi\colon X\to Y$ be a projective morphism of complex analytic 
spaces. 
Let $U$ be an open subset of $Y$ and 
let $W$ be a Stein compact subset of $Y$ such that 
$\Gamma (W, \mathcal O_Y)$ is noetherian with $U\subset W$. 
Assume that $K_X+\Delta$ is $\pi$-semiample. 
Let $P$ be an arbitrary point of $U$. 
Then there exists a semianalytic Stein open neighborhood 
$U_P$ of $P$ and a positive integer $m$ such that 
admissible sections generate $\mathcal O_X(m(K_X+\Delta))$ over 
$U_P$. 
\end{lem}

\begin{proof}
Let $\nu\colon X^\nu \to X$ be the normalization. 
We write $K_{X^\nu}+\Theta=\nu^*(K_X+\Delta)$ as in 
Definition \ref{c-def2.25}. By Lemma \ref{c-lem2.29}, any 
admissible section on $X^\nu$ over $U_P$ descends to an 
admissible section on $X$ over $U_P$. 
Hence, by replacing $X$ with its normalization 
$X^\nu$, we may assume that $X$ is normal.

By \cite[Theorems 1.21 and 1.27]{fujino-bchm}, we can take a dlt blow-up $\varphi\colon X'\to X$ with $K_{X'}+\Delta'=\varphi^*(K_X+\Delta)$. 
We may assume that $\varphi$ is small, maps every log canonical center of $(X', \Delta')$ bimeromorphically onto a log canonical center of $(X, \Delta)$, and induces an isomorphism over general points of each log canonical center of $(X, \Delta)$. 
By replacing $(X, \Delta)$ with $(X', \Delta')$, we may further assume that $X$ is $\mathbb Q$-factorial over $W$.

By Lemma \ref{p-lem4.8} and Proposition \ref{p-prop4.9}, it suffices to prove this lemma for $(S, \Delta_S)$, where $S:=\lfloor \Delta\rfloor$ and $K_S+\Delta_S:=(K_X+\Delta)|_S$. 
By repeating this reduction process finitely many times, we can reduce the problem to the case where $(X, \Delta)$ is kawamata log terminal. 
In this case, any section is preadmissible (see Remark \ref{c-rem2.28}). 
Thus, by Lemma \ref{p-lem4.8}, we obtain the desired result.
\end{proof}

Let us prove Theorem \ref{b-thm1.1}, which is 
one of the main results of the present paper. 

\begin{proof}[Proof of Theorem \ref{b-thm1.1}]
We take an arbitrary point $P\in W$. 
Since $W$ is compact, it suffices to prove that 
there exists a positive integer $m_P$ such that 
$\mathcal O_X(m_P(K_X+\Delta))$ is $\pi$-generated 
over some open neighborhood of $P$. 
We take a semianalytic Stein open neighborhood 
$U_P$ of $P$ and a Stein compact subset $W_P$ of $Y$ with 
$U_P\subset W_P$ such that $\Gamma (W_P, \mathcal O_Y)$ is noetherian.  
Let $\nu\colon X^\nu\to X$ be the normalization with $K_{X^\nu}+\Theta
:=\nu^*(K_X+\Delta)$. 
By \cite[Theorems 1.21 and 1.27]{fujino-bchm}, after shrinking $Y$ around $W_P$ suitably, 
we take a dlt blow-up $\alpha\colon 
\widetilde X\to X^\nu$ with $K_{\widetilde X}+
\widetilde{\Delta}:=\alpha^*(K_{X^\nu}+\Theta)$ such that 
$\widetilde X$ is $\mathbb Q$-factorial over $W_P$ and $(\widetilde X, 
\widetilde \Delta)$ is divisorial log terminal. 
We consider $\widetilde \pi:=\pi\circ 
\nu\circ \alpha \colon \widetilde X\to Y$. 
If necessary, we replace $U_P$ with a smaller semianalytic Stein open 
neighborhood of $P$. 
Then, by Lemma \ref{p-lem4.10}, there exists a semianalytic 
Stein open neighborhood $U_P$ and a positive integer $m_P$ such that 
admissible sections generate $\mathcal O_{\widetilde X}(m_P(K_{\widetilde X}+
\widetilde{\Delta}))$ over $U_P$. By Lemma \ref{c-lem2.29}, 
any admissible section over $U_P$ descends to a section of 
$\mathcal O_X(m_P(K_X+\Delta))$ over $U_P$. 
Thus $\mathcal O_X(m_P(K_X+\Delta))$ is $\pi$-generated over $U_P$. 
As we mentioned above, since $W$ is compact, 
we can take an open neighborhood $U$ of $W$ and a 
divisible positive integer $m$ such that 
$\mathcal O_X(m(K_X+\Delta))$ is $\pi$-generated over $U$. 
We finish the proof of Theorem \ref{b-thm1.1}. 
\end{proof}

The following elementary lemma shows that \cite[Theorem 2]{hacon-xu} 
can be deduced from Theorem~\ref{b-thm1.1}. Consequently, Koll\'ar's 
gluing theory in \cite{kollar} is not required for the proof 
of \cite[Theorem 2]{hacon-xu}.

\begin{lem}\label{p-lem4.11}
Let $\pi\colon X\to Y$ be a proper morphism 
of algebraic schemes defined over $\mathbb C$ and 
let $\mathcal L$ be a line bundle on $X$. 
Let $U$ be a nonempty Euclidean open subset of $Y$. 
Assume that $\mathcal L$ is $\pi$-generated 
over $U$. 
Then there exists a Zariski open subset $V$ of $Y$ such 
that $\mathcal L$ is $\pi$-generated over $V$ with $U\subset V$. 
\end{lem}

\begin{proof}
Let $\mathcal C$ be the cokernel of $\pi^*\pi_*\mathcal L\to \mathcal L$. 
We put $V:=Y\setminus \pi(\Supp \mathcal C)$. 
Then, by definition, $V$ is a Zariski open subset with $U\subset V$ and 
$\mathcal L$ is $\pi$-generated over $V$. 
\end{proof}

\section{Freeness for nef and log abundant log canonical bundles}\label{p-sec5}

In this section, we first prove Theorem \ref{b-thm1.4}. 
As a straightforward application of Theorem \ref{b-thm1.4}, 
we then establish Theorem \ref{b-thm1.10}. 
We also give proofs of Theorem \ref{b-thm1.5}, 
Theorem \ref{b-thm1.7}, and Corollary \ref{b-cor1.11}. 
The proof of Theorem \ref{b-thm1.4} begins with 
Theorem \ref{p-thm5.1}, which is well known in 
the algebraic setting (see \cite{fujino-basepoint-free}). 
Once Theorem \ref{p-thm5.1} is established, the proof of 
Theorem \ref{b-thm1.4} follows without much difficulty.

\begin{thm}\label{p-thm5.1} 
Let $(X, \Delta)$ be an irreducible 
divisorial log terminal pair and let $\pi\colon 
X\to Y$ be a projective morphism of complex analytic spaces. 
Assume that $K_X+\Delta$ is $\mathbb Q$-Cartier and 
is $\pi$-nef and $\pi$-abundant over $Y$. 
We further assume that $K_S+\Delta_S$ is $\pi$-semiample, 
where $S:=\lfloor \Delta\rfloor$ and $K_S+\Delta_S:=(K_X+\Delta)|_S$. 
Let $W$ be a compact subset of $Y$. 
Then there exists a positive integer $m$ such that 
$\mathcal O_X(m(K_X+\Delta))$ is $\pi$-generated 
over some open neighborhood of $W$. 
\end{thm}

\begin{proof}
We can modify the argument in \cite[Section 6]{fujino-basepoint-free} for 
our complex analytic setting. 
Since the Kawamata--Viehweg vanishing theorem 
holds for projective morphisms of complex analytic spaces, 
we can generalize \cite[Theorem 6.1]{fujino-basepoint-free}, 
which is a slight generalization of the Kawamata--Shokurov 
basepoint-free theorem,  
for our complex analytic setting. 
By \cite[Theorem 21.4]{fujino-bchm}, 
which is a kind of canonical bundle formula, and 
the argument in Step 2 in the proof of \cite[Theorem 23.2]{fujino-bchm}, 
we can prove a complex analytic generalization of 
\cite[Theorem 6.2]{fujino-basepoint-free}. 
Therefore, we see that the desired statement holds 
(see also \cite[Theorem 1.1]{fujino-basepoint-free}). 
\end{proof}

Let us prove Theorem \ref{b-thm1.4}. 

\begin{proof}[Proof of Theorem \ref{b-thm1.4}]
Let $P$ be an arbitrary point of $W$. 
Since $W$ is compact, it suffices to prove that there 
exist a positive integer $m_P$ and an open neighborhood 
$U_P$ of $P$ such 
that 
$\mathcal O_X(m_P(K_X+\Delta))$ is $\pi$-generated over $U_P$. 
By Theorem \ref{b-thm1.1}, we may assume that 
$X$ is normal. By taking a Stein 
compact subset $W_P$ such that $P\in W_P$ and 
$\Gamma (W_P, \mathcal O_Y)$ is noetherian. 
By \cite[Theorems 1.21 and 1.27]{fujino-bchm}, 
after shrinking $Y$ around $W_P$ suitably, 
we take a dlt blow-up and may assume that 
$(X, \Delta)$ is divisorial log terminal. 
By induction on dimension, we may assume that 
$K_S+\Delta_S$ is $\pi$-semiample over some open 
neighborhood of $P$, 
where $S:=\lfloor \Delta\rfloor$ and 
$K_S+\Delta_S:=(K_X+\Delta)|_S$. 
Hence, by Theorem \ref{p-thm5.1}, 
we obtain the desired statement. 
We finish the proof. 
\end{proof}

The following proof is essentially due to 
Kenta Hashizume (see \cite[Lemma 3.4]{hashizume}). 

\begin{proof}[Proof of Theorem \ref{b-thm1.5}]
We can freely shrink $Y$ around $W$ suitably and 
always assume that $Y$ is Stein. 
By taking a dlt blow-up 
(see \cite[Theorems 1.21 and 1.27]{fujino-bchm}), 
we may assume that $(X, \Delta)$ is divisorial log terminal and 
is $\mathbb Q$-factorial over $W$. 
By induction, we may assume that 
$K_S+\Delta_S:=(K_X+\Delta)|_S$ is $\pi$-semiample 
over some open neighborhood 
of $L$ for every log canonical center $S$ of 
$(X, \Delta)$. 
By applying the argument in the proof of \cite[Lemma 3.4]{hashizume}, 
we can write $K_X+\Delta=\sum _i r_i(K_X+\Delta_i)$ 
such that $(X, \Delta_i)$ is divisorial log terminal, 
$K_X+\Delta_i$ is $\mathbb Q$-Cartier, 
$r_i$ is a positive real number, and $K_X+\Delta_i$ is $\pi$-nef 
and $\pi$-log abundant over some open neighborhood 
of $L$ for every $i$. 
Hence, by Theorem \ref{b-thm1.4}, 
there exists a positive 
integer $m_i$ such that 
$\mathcal O_X(m_i(K_X+\Delta_i))$ is $\pi$-generated 
over some open neighborhood of $L$. 
Hence, $K_X+\Delta$ is $\pi$-semiample 
over some open neighborhood of $L$. 
This is what we wanted. 
\end{proof}

Theorem \ref{b-thm1.7} is almost obvious by Theorem 
\ref{b-thm1.5} and \cite[Theorem 1.2]{enokizono-hashizume-mmp}. 

\begin{proof}[Proof of Theorem \ref{b-thm1.7}]
We take an arbitrary point $P\in Z$. 
Then it suffices to prove the existence of 
a log canonical model of $(X, \Delta)$ over some open neighborhood 
of $P$. We take $P\in U_1\subset W_1\subset U_2\subset W_2$, 
where $U_i$ is a Stein open subset of $Z$ for $i=1, 2$ and 
$W_i$ is a Stein compact subset of $Z$ such that 
$\Gamma (W_i, \mathcal O_Z)$ is noetherian for $i=1, 2$. 
Throughout this proof, we can freely shrink $Z$ around $W_2$ suitably. 
Since $-(K_X+\Delta)$ is $\varphi$-ample, 
we can take an effective $\mathbb R$-divisor $A$ on $X$ 
such that $K_X+\Delta+A\sim _{\mathbb R, \varphi}0$ and 
that 
$(X, \Delta+A)$ is log canonical. 
By \cite[Theorems 1.21 and 1.27]{fujino-bchm}, 
we take a dlt blow-up $p\colon (X', \Delta')\to (X, \Delta)$ over 
some open neighborhood of $W_2$. We note 
that $(X', \Delta'+A')$ is log canonical 
with $K_{X'}+\Delta'+A'\sim _{\mathbb R, 
\varphi'}0$, where $A':=p^*A$ and $\varphi':=\varphi\circ p\colon X'\to Z$. 
It suffices to construct a log canonical model 
of $(X', \Delta')$ over some open neighborhood of $P$. 
By \cite[Theorem 1.2]{enokizono-hashizume-mmp}, 
after finitely many flips and divisorial contractions, 
we finally obtain $(X'', \Delta'')$ 
over some open neighborhood of $W_2$ such that $K_{X''}+\Delta''$ is 
nef over $W_2$. 
By construction, $K_{X''}+\Delta''+A''\sim _{\mathbb R, \varphi''}0$ 
holds, where $A''$ is the pushforward of $A'$ on $X''$ and 
$\varphi''\colon X''\to Z$ is the structure morphism. 
Thus, by \cite[Theorem 6.1]{gongyo-mmp}, we can check that 
$K_{X''}+\Delta''$ is $\varphi''$-nef and 
$\varphi''$-log abundant with 
respect to $(X'', \Delta'')$ over $U_2$  
(see also \cite[Remark 3.7]{hashizume3}). 
Therefore, by Theorem \ref{b-thm1.5}, 
$K_{X''}+\Delta''$ is $\varphi''$-semiample over some 
open neighborhood of $P$. 
This means that $(X', \Delta')$ has a log canonical model 
over some open neighborhood of $P$. 
This is what we wanted. We finish the proof.     
\end{proof}

We prove Theorem \ref{b-thm1.10} as an application of 
Theorem \ref{b-thm1.4}. 

\begin{proof}[Proof of Theorem \ref{b-thm1.10}]
Let $P$ be an arbitrary point of $W$. 
Since $W$ is compact, it suffices to 
prove that there exist a positive integer $m_P$ and 
an open neighborhood $U_P$ of $P$ 
such that $\mathcal O_X(m_P(K_X+\Delta))$ is $\pi$-generated 
over $U_P$. 
From now, we will freely shrink $Y$ around $P$. 
By \cite[Theorems 1.21 and 1.27]{fujino-bchm}, we take a dlt blow-up. 
Thus we may assume that $(X, \Delta)$ is divisorial 
log terminal. 
Let $S$ be a log canonical stratum of $(X, \Delta)$ 
with $K_S+\Delta_S:=(K_X+\Delta)|_S$. 
It is obvious that $K_S+\Delta_S$ is $\pi$-nef. 
By applying Conjecture \ref{b-conj1.9} to 
an analytically sufficiently general fiber $F$ of $S\to \pi(S)$, 
we see that $K_S+\Delta_S$ is $\pi$-nef and $\pi$-abundant. 
This means that $K_X+\Delta$ is $\pi$-nef and $\pi$-log 
abundant with respect to $(X, \Delta)$. 
Hence, by Theorem \ref{b-thm1.4}, 
we obtain $m_P$ such that $\mathcal O_X(m_P(K_X+\Delta))$ 
is $\pi$-generated over some open neighborhood 
of $P$. 
We finish the proof.  
\end{proof}

Let us prove Corollary \ref{b-cor1.11}, which is an easy application of 
Theorem \ref{b-thm1.10}. 

\begin{proof}[Proof of Corollary \ref{b-cor1.11}]
We can freely shrink $Y$ around $W$. 
By using Shokurov's polytope (see \cite{fujino-bchm}), 
we can write $K_X+\Delta=\sum _i r_i(K_X+\Delta_i)$ such that 
$(X, \Delta_i)$ is log canonical, $K_X+\Delta_i$ is 
$\mathbb Q$-Cartier, 
$r_i$ is a positive real number, and 
$K_X+\Delta_i$ is $\pi$-nef over $W$ for every $i$. 
In particular, $K_X+\Delta_i$ is $\pi$-nef over 
$U$ for every $i$. 
Then, by Theorem \ref{b-thm1.10}, 
there exists a positive integer $m_i$ such that 
$\mathcal O_X(m_i(K_X+\Delta_i))$ is $\pi$-generated 
over some open neighborhood of $L$ for every $i$. 
This implies that $K_X+\Delta$ is $\pi$-semiample 
over some open neighborhood of $L$. 
We finish the proof. 
\end{proof}

In any case, by Theorem \ref{b-thm1.10} and Corollary \ref{b-cor1.11}, 
the abundance conjecture for projective morphisms of complex analytic spaces 
is completely reduced to the original abundance conjecture 
for projective varieties (see Conjecture \ref{b-conj1.9}). 
We conclude this section with an important conjecture.

\begin{conj}\label{p-conj5.2}
Let $\pi\colon X\to Y$ be a projective 
surjective morphism of normal complex varieties and 
let $(X, \Delta)$ be a log canonical pair. 
Let $W$ be a compact subset of $Y$. 
Assume that $K_X+\Delta$ is $\pi$-nef over $W$. 
Then $K_X+\Delta$ is $\pi$-nef over some 
open neighborhood of $W$. 
\end{conj}

If Conjecture \ref{p-conj5.2} 
holds, 
then we can prove that 
$K_X+\Delta$ is $\pi$-semiample 
over some open neighborhood 
of $W$ in 
Theorem \ref{b-thm1.5} and Corollary \ref{b-cor1.11}. 
In the case $\dim X = 2$, Conjecture \ref{p-conj5.2} has 
been completely resolved 
in \cite{moriyama}.

\section{Dlt blow-ups and some applications}\label{y-sec6}

In this section, we discuss {\em{dlt blow-ups}} 
(see also \cite[Theorems 1.21 and 1.27]{fujino-bchm}) and some applications. 
This section heavily relies on the minimal model 
program established in \cite{enokizono-hashizume-mmp} and 
hence in \cite{fujino-bchm}. 

We begin with the following statement. 
In the original algebraic setting, the reader can find it in  
\cite[Proposition 3.3.1]{hacon-mckernan-xu}. 

\begin{thm}[Dlt blow-ups for log canonical pairs]\label{y-thm6.1} 
Let $X$ be a normal complex variety and let $\Delta$ be an effective 
$\mathbb R$-divisor on $X$ such that $(X, \Delta)$ is log canonical. 
Let $W$ be a Stein compact subset of $X$ such that $\Gamma (W, \mathcal O_X)$ 
is noetherian. 
Then, after shrinking $X$ around $W$ suitably, 
we can construct a projective bimeromorphic morphism 
$f\colon Z\to X$ from a normal complex variety $Z$ with the following 
properties: 
\begin{itemize}
\item[{\em{(i)}}] $Z$ is $\mathbb Q$-factorial over $W$, 
\item[{\em{(ii)}}] $a(E, X, \Delta)= -1$ for every $f$-exceptional 
divisor $E$ on $Z$, and 
\item[{\em{(iii)}}] $(Z, \Delta_Z)$ is divisorial 
log terminal, where $K_Z+\Delta_Z=f^*(K_X+\Delta)$. 
\end{itemize} 
Moreover, let $S$ be an irreducible component of $\Delta$ and let 
$T$ be the strict transform of $S$ on $Z$. 
Then we can make $f\colon Z\to X$ satisfy: 
\begin{itemize}
\item[{\em{(iv)}}] there exists an effective $f$-exceptional $\mathbb Q$-divisor 
$F$ on $Z$ with 
$f(F)\subset S$ such that 
$-T-F$ is $f$-nef over $W$. 
\end{itemize}
\end{thm}

The following proof is well known for algebraic varieties. 
In the complex analytic setting, it suffices to 
use \cite[Theorem 1.2]{enokizono-hashizume-mmp}. 

\begin{proof}[Proof of Theorem \ref{y-thm6.1}]
By \cite[Theorem 1.27]{fujino-bchm}, after shrinking $X$ around $W$ suitably, 
we can take a 
projective bimeromorphic morphism $g\colon V\to X$ from a normal complex 
variety $V$ with $K_V+\Delta_V=g^*(K_X+\Delta)$ satisfying 
(i), (ii), and (iii). 
Let $T_V$ be the strict transform of $S$ on $V$. 
By \cite[Theorem 1.2]{enokizono-hashizume-mmp}, 
after running a suitable minimal model program with ample 
scaling over $X$ around $W$, we 
obtain a minimal model $(V', \Delta_{V'}-T_{V'})$ 
of $(V, \Delta_V-T_V)$ over some open neighborhood 
of $W$. 
Let $\phi\colon V\dashrightarrow V'$ be the induced bimeromorphic map, 
and set $\Delta_{V'}:=\phi_*\Delta_V$ and $T_{V'}:=\phi_*T_V$. 
Let $g_{V'}\colon V'\to X$ be the induced morphism. 
Then we have $K_{V'}+\Delta_{V'}\sim _{\mathbb R, g_{V'}}0$. Moreover, 
$(V', \Delta_{V'})$ is log canonical, 
and $K_{V'}+(\Delta_{V'}-T_{V'})\sim _{\mathbb R, g_{V'}}-T_{V'}$ 
is $g_{V'}$-nef 
over $W$. 
By \cite[Theorem 1.27]{fujino-bchm} again, 
we can take a dlt blow-up $h\colon (Z, \Delta_Z)\to (V', \Delta_{V'})$ 
with $K_Z+\Delta_Z=h^*(K_{V'}+\Delta_{V'})$ after shrinking $X$ around $W$. 
Then we can write $h^*(-T_{V'})=-T-F$, where $T$ is the strict transform of $T_{V'}$ on $Z$. 
By construction, $F$ is an 
effective $f$-exceptional $\mathbb Q$-divisor with $f(F)\subset S$, where 
$f:=g_{V'}\circ h\colon Z\to X$. 
By construction, we can easily verify that $(Z, \Delta_Z)$ 
satisfies (i), (ii), (iii), and (iv). 
This completes the proof. 
\end{proof}

As a straightforward application of Theorem \ref{y-thm6.1}, 
we consider Theorem \ref{y-thm6.2}, which is a slight 
generalization of \cite[Theorem 2.1.6]{fujino-acc}. 
Note that \cite[Theorem 2.1.6]{fujino-acc} follows easily from 
Theorem \ref{y-thm6.2}. For further details, 
see \cite[Proof of (1.1)]{hacon-mckernan-xu}. 
Notably, Theorem \ref{y-thm6.2} can be viewed as 
a complex analytic generalization of \cite[Theorem 1.4]{hacon-mckernan-xu}.

\begin{thm}[ACC for the log canonical 
thresholds for complex analytic spaces]\label{y-thm6.2}
Fix a positive integer $n$ and a set $I\subset [0, 1]$ which satisfies 
the {\em{DCC}}. Then there is a finite subset $I_0\subset I$ with the following 
properties: 

If $X$ is a normal complex variety and let $\Delta$ be an effective 
$\mathbb R$-divisor on $X$ such that $K_X+\Delta$ is $\mathbb R$-Cartier and 
that 
\begin{itemize}
\item[{\em{(1)}}] $\dim X=n$, 
\item[{\em{(2)}}] $(X, \Delta)$ is log canonical, 
\item[{\em{(3)}}] the coefficients of $\Delta$ belong to $I$, and 
\item[{\em{(4)}}] there exists a log canonical center $C\subset X$ which 
is contained in every component of $\Delta$, 
\end{itemize} 
then the coefficients of $\Delta$ belong to $I_0$. 
\end{thm}

\begin{proof}[Proof of Theorem \ref{y-thm6.2}]
We note that \cite[Lemma 5.1]{hacon-mckernan-xu} also 
holds in the complex analytic setting by virtue of 
Theorem \ref{y-thm6.1}. For details, we refer 
the reader to the proof of \cite[Lemma 5.1]{hacon-mckernan-xu}. 
Consequently, the arguments in \cite[Section 5]{hacon-mckernan-xu} 
apply equally well in the complex analytic setting. 
Therefore, Theorem \ref{y-thm6.2} follows from the 
ACC for numerically trivial pairs (see \cite[Theorem 1.5]{hacon-mckernan-xu}).
\end{proof}

Theorem \ref{b-thm1.8} is significantly 
more profound than Theorem \ref{y-thm6.1}.
It can be regarded as a complex analytic version 
of \cite[Lemma 3.5]{fujino-hashizume}.
For related results in the original algebraic setting, 
see \cite{fujino-hashizume}. Let us prove Theorem \ref{b-thm1.8}. 

\begin{proof}[Proof of Theorem \ref{b-thm1.8}] 
Throughout this proof, we may freely shrink $X$ around $W$ without 
explicit mention. 
By \cite[Theorem 1.27]{fujino-bchm}, 
we can take a 
projective bimeromorphic morphism 
$g\colon V\to X$ from a normal complex variety $V$ with 
$K_V+\Delta_V=g^*(K_X+\Delta)$ satisfying (i), (ii), and (iii). 
We consider $K_V+\Delta^\dag_V=g^*(K_X+\Delta)-G_V$, where $\Delta^\dag_V:=
\Delta^{\leq 1}_V+\Supp \Delta^{>1}_V$. 
We take a rational number $\varepsilon $ with 
$0<\varepsilon \ll 1$. 
Then we consider 
\[
K_V+\Delta^\dag_V-\varepsilon G_V=g^*(K_X+\Delta)-(1+\varepsilon) G_V. 
\] 
Note that $(V, \Delta^\dag_V-\varepsilon G_V)$ is 
divisorial log terminal and $K_V+\Delta^\dag_V-\varepsilon G_V$ is 
$g$-log abundant with respect to $(V, \Delta^\dag_V-\varepsilon G_V)$ 
(see \cite[Theorem 6.1]{gongyo-mmp}). 
We take a general $g$-ample effective $\mathbb Q$-divisor $A$ on $V$ such 
that $(V, \Delta^\dag_V-\varepsilon G_V+A)$ is log 
canonical and $K_V+\Delta^\dag_V-\varepsilon G_V+A$ is $g$-nef over 
$W$. 
We run a $(K_V+\Delta^\dag_V-\varepsilon G_V)$-minimal 
model program over $X$ around $W$ with scaling of $A$ 
starting from $(V_0, \Delta^\dag_{V_0}-\varepsilon G_{V_0}):=
(V, \Delta^\dag_V-\varepsilon G_V)$. Then we have a sequence of 
flips and divisorial contractions:  
\[
V_0\overset{\phi_0}{\dashrightarrow} 
V_1 \overset{\phi_1}{\dashrightarrow} 
\cdots \overset{\phi_{i-1}}{\dashrightarrow} 
V_i 
\overset{\phi_i}{\dashrightarrow}  \cdots 
\] 
with 
\[
1\geq \lambda_0\geq \lambda_1\geq \cdots
\] 
such that $K_{V_i}+
\Delta^\dag_{V_i}-\varepsilon G_{V_i}+\lambda_iA_i$ is $g_{V_i}$-nef 
over $W$, where 
$g_{V_i}\colon V_i\to X$, 
$\Delta^\dag_{V_i}:=(\phi_{i-1})_*\Delta^\dag_{V_{i-1}}$, 
$G_{V_i}:=(\phi_{i-1})_*G_{V_{i-1}}$, and 
$A_i:=(\phi_{i-1})_*A_{i-1}$ for every $i\geq 1$. 
We put $\lambda_{\infty}:=\lim _{i\to \infty} \lambda_i$. 
Then we obtain $\lambda _{\infty}=0$. 
For the details, see, for example, the proof of 
\cite[Lemma 13.7]{fujino-bchm}. 
On the other hand, since $K_{V_i}+\Delta^\dag_{V_i}
+G_{V_i}\sim _{\mathbb R, 
g_{V_i}}0$, 
$K_{V_i}+\Delta^\dag_{V_i}-\varepsilon G_{V_i}$ is $g_{V_i}$-log 
abundant with respect to $(V_i, \Delta^\dag_{V_i}-\varepsilon G_{V_i})$ 
for every $i\geq 0$ by 
\cite[Theorem 6.1]{gongyo-mmp}. 
Then, by \cite[Theorem 1.3]{enokizono-hashizume-mmp}, 
it terminates at a minimal model 
$(V', \Delta^\dag_{V'}-\varepsilon G_{V'})$. 
We note that $\phi\colon V\dashrightarrow V'$, 
$\phi_*\Delta^\dag_V=\Delta^\dag_{V'}$, and $\phi_*G_V=G_{V'}$. 
We put $g_{V'}\colon V'\to X$. 
Since $K_V+\Delta^\dag_V\sim _{\mathbb R, g}-G_V$ and 
$K_V+\Delta^\dag_V-\varepsilon G_V\sim _{\mathbb R, g} -(1+\varepsilon) G_V$, 
the above minimal model program is also a $(K_V+\Delta^\dag_V)$-minimal model 
program. 
In particular, $(V', \Delta^\dag_{V'})$ is 
a divisorial log terminal pair. 
We put $(Z, \Delta_Z):=(V', \Delta_{V'})$ and $f:=g_{V'}$. 
Then it is easy to see that it satisfies (i), (ii), (iii), and (iv). Since 
\[
-(1+\varepsilon)G\sim _{\mathbb R, f}K_Z+\Delta^\dag_Z-\varepsilon G
\]  
and $K_Z+\Delta^\dag_Z-\varepsilon G$ is $f$-nef and 
$f$-log abundant with respect to $(Z, \Delta^\dag_Z-\varepsilon G)$, 
it is $f$-semiample over some open neighborhood of $L$ 
by Theorem \ref{b-thm1.5}. 
This implies that $-G$ is $f$-semiample over some open neighborhood 
of $L$.   
This is (v). 
We finish the proof of Theorem 
\ref{b-thm1.8}. 
\end{proof}

We can quickly recover the log canonical inversion of adjunction 
(see \cite[Theorem 1.1]{fujino-inversion}) from Theorem \ref{b-thm1.8}. 

\begin{thm}[{Inversion of adjunction for log canonicity, 
see \cite[Theorem 1.1]{fujino-inversion}}]\label{y-thm6.3}
Let $X$ be a normal complex variety and let $S+B$ be an 
effective $\mathbb R$-divisor on $X$ such that 
$K_X+S+B$ is $\mathbb R$-Cartier, $S$ is reduced, and 
$S$ and $B$ have no common irreducible components. 
Let $\nu\colon S^\nu\to S$ be the normalization with 
$K_{S^\nu}+B_{S^\nu}=\nu^*(K_X+S+B)$, where 
$B_{S^\nu}$ denotes Shokurov's different. Then 
$(X, S+B)$ is log canonical in a neighborhood 
of $S$ if and only if $(S^\nu, B_{S^\nu})$ is log 
canonical.  
\end{thm}

For the sake of completeness, we provide 
a proof of Theorem \ref{y-thm6.3} based on 
Theorem \ref{b-thm1.8}.
In the proof of Theorem \ref{y-thm6.3} below, we make 
use of (iv) of Theorem \ref{b-thm1.8}, but (v) of 
Theorem \ref{b-thm1.8} is not needed.

\begin{proof}[Proof of Theorem \ref{y-thm6.3}]
If $(X, S+B)$ is log canonical in a neighborhood of $S$, 
then we can easily check that $(S^\nu, B_{S^\nu})$ is log canonical. 
From now, we will prove that $(X, S+B)$ is log canonical 
in a neighborhood of $S$ under the assumption that 
$(S^\nu, B_{S^\nu})$ is log canonical. 
We take an arbitrary point $P\in S$. It suffices 
to prove that $(X, S+B)$ is log canonical around $P$. 
We take a Stein compact subset $W$ of $X$ such that 
$\Gamma (W, \mathcal O_X)$ is noetherian and 
an open neighborhood $U$ of $P$ such that 
$P\in U\subset W$. 
Let $f\colon Z\to X$ be a dlt blow-up after shrinking $X$ around 
$W$ with $K_Z+\Delta_Z=f^*(K_X+\Delta)$ as in 
Theorem \ref{b-thm1.8}. 
Note that $K_Z+\Delta^\dag_Z=f^*(K_X+\Delta)-G$ is $f$-nef over $W$. 
We also note that $\Nlc (Z, \Delta_Z)=\Supp G$. 
Since $-G$ is $f$-nef over $W$, 
$\Nlc(Z, \Delta_Z)=f^{-1}\left(\Nlc(X, \Delta)\right)$ holds 
over $U$. 
Let $T$ be the strict transform of $S$ on $Z$. 
Since $(S^\nu, B_{S^\nu})$ is log canonical, 
$T\cap \Supp G=\emptyset$. 
This implies that $S\cap \Nlc(X, \Delta)=\emptyset$ on $U$. 
Hence $(X, \Delta)$ is log canonical around $S$ on $U$. 
We finish the proof. 
\end{proof}

\section{Appendix:~On Fano contractions}\label{f-sec7}

In this appendix, we explicitly state some results on extremal Fano contractions. 
Here, by an extremal Fano contraction, we mean a contraction morphism associated to an extremal ray whose target has strictly smaller dimension than its source. 
While the results explained in this section are well known in the original algebraic setting, they do not seem to be explicitly stated in the literature within the complex analytic framework. 
For the convenience of the reader, we explain them below. 
Although these results are not strictly indispensable, they are used in Section \ref{p-sec4} of this paper.

\medskip

First, let us recall some basic definitions.

\begin{defn}[{Divisorial sheaves, see \cite[Definition 5.4.1]{fujino-cone-contraction}}]\label{f-def7.1}
Let $X$ be a normal complex variety. 
A coherent sheaf $\mathcal L$ on $X$ is said to be {\em{divisorial}} if it is reflexive of rank one. 
\end{defn}

In the complex analytic setting, a divisorial sheaf does not necessarily correspond to a Weil divisor. Thus, we need the following definition.

\begin{defn}[{Strong $\mathbb Q$-factoriality, see \cite[Definition 5.4.4]{fujino-cone-contraction}}]\label{f-def7.2}
Let $\mathcal L$ be a divisorial sheaf on a normal complex variety $X$. 
Let $x$ be a point of $X$. 
If there exists a positive integer $m_x$ such that 
\[
\mathcal L^{[m_x]}:=\left(\mathcal L^{\otimes m_x}\right)^{**}
\] 
is a line bundle on some open neighborhood of $x\in X$, then $\mathcal L$ is said to be {\em{$\mathbb Q$-Cartier at $x$}}. 

Let $\pi\colon X\to Y$ be a morphism of complex analytic spaces such that 
$X$ is a normal complex variety and let $W$ be a subset of $Y$. If every divisorial sheaf defined over some open neighborhood of $\pi^{-1}(W)$ is $\mathbb Q$-Cartier at any point of $\pi^{-1}(W)$, then we say that $X$ is {\em{strongly $\mathbb Q$-factorial over $W$}}. 
\end{defn}

Strong $\mathbb Q$-factoriality is preserved under extremal Fano contractions.

\begin{lem}[Fano contractions]\label{f-lem7.3}
Let $\pi\colon X\to Y$ be a projective morphism of complex analytic spaces, and let $W$ be a compact subset of $Y$ such that the dimension of $N^1(X/Y; W)$ is finite. 
Let $(X, \Delta)$ be a log canonical pair. 
Let $\varphi_R\colon X\to Z$ be the Fano contraction, that is, $\dim Z<\dim X$, associated to a $(K_X+\Delta)$-negative extremal ray $R$ of $\overline{\NE}(X/Y; W)$. 
Assume that $X$ is strongly $\mathbb Q$-factorial over $W$. 
Then $Z$ is strongly $\mathbb Q$-factorial over $W$. 
In particular, $K_Z$ is $\mathbb Q$-Cartier after shrinking $Y$ around $W$ suitably. 
\end{lem}

Although the following proof is more or less well known to experts (see, for example, \cite[Corollary 3.18]{kollar-mori} and the proof of \cite[Lemma 5.4.9]{fujino-cone-contraction}), we describe it in detail for the reader's convenience. 

\begin{proof}[Proof of Lemma \ref{f-lem7.3}] 
Throughout the proof, we will freely shrink $Y$ around $W$ without explicit mention. 
Let $\mathcal L$ be a divisorial sheaf on $Z$. 
We put $\mathcal M:=\left(\varphi^*_R\mathcal L\right)^{**}$. 
Then $\mathcal M$ is a divisorial sheaf on $X$. 
Since $X$ is strongly $\mathbb Q$-factorial over $\pi^{-1}(W)$, there exists a positive integer $m$ such that $\mathcal M^{[m]}$ is a line bundle on $X$. 
We see that $\mathcal M^{[m]}\cdot C=0$ for every $[C]\in R$. 
By the cone and contraction theorem (see, for example, \cite[Theorem 1.1.6]{fujino-cone-contraction}), there exists a line bundle $\mathcal N$ on $Z$ such that $\mathcal M^{[m]}\simeq \varphi^*_R\mathcal N$. 
This implies that $\mathcal L^{[m]}\simeq \mathcal N$. 
Hence $\mathcal L$ is $\mathbb Q$-Cartier. 
Thus $Z$ is strongly $\mathbb Q$-factorial over $W$. 
In particular, $K_Z$ is $\mathbb Q$-Cartier. 
\end{proof}

In the settings we frequently consider, the notions of $\mathbb Q$-factoriality and strong $\mathbb Q$-factoriality coincide.

\begin{lem}\label{f-lem7.4}
Let $\pi\colon X\to Y$ be a projective morphism of complex analytic spaces and let $W$ be a Stein compact subset of $Y$. 
Then $X$ is $\mathbb Q$-factorial over $W$ if and only if $X$ is strongly $\mathbb Q$-factorial over $W$. 
\end{lem}

\begin{proof}
Let $\mathcal L$ be a divisorial sheaf defined over some open neighborhood of $\pi^{-1}(W)$. 
Then, after replacing $Y$ with a suitable Stein open neighborhood of $W$, we can always find a Weil divisor $D$ on $X$ such that $\mathcal L \simeq \mathcal O_X(D)$. For details, see \cite[Remark 2.43]{fujino-bchm}. This implies the desired equivalence.
\end{proof}

The following theorem concerning the singularities of the target of a Fano contraction is nontrivial and deeper than Lemma \ref{f-lem7.3}. 
However, it is more or less a recognized fact in the algebraic case.

\begin{thm}\label{f-thm7.5}
Let $f\colon X\to Y$ be a projective surjective morphism of complex analytic spaces with $f_*\mathcal O_X\simeq \mathcal O_Y$. 
Assume that $(X, \Delta)$ is kawamata log terminal with $K_X+\Delta\sim _{\mathbb Q, f} 0$. 
We further assume that $K_Y$ is $\mathbb Q$-Cartier. 
Then $Y$ is kawamata log terminal. 
\end{thm}

This theorem is nothing but a generalization of \cite[Theorem 1.2]{fujino-application} to the complex analytic setting.

\begin{proof}[Sketch of Proof of Theorem \ref{f-thm7.5}] 
Since $K_Y$ is assumed to be $\mathbb Q$-Cartier, it suffices to take an arbitrary point $y\in Y$ and show that there exists an effective $\mathbb Q$-divisor $\Gamma$ on $Y$ such that the pair $(Y, \Gamma)$ is kawamata log terminal in a neighborhood of $y$. 
Therefore, we may fix a point $y\in Y$ and replace $Y$ with a sufficiently small Stein open neighborhood of $y$. 
Note that we are allowed to repeat this replacement finitely many times during the argument if necessary. 
Under this localized setting, it is possible to run the proof of \cite[Theorem 1.2]{fujino-application} within the complex analytic framework. 
In \cite{fujino-application}, Kawamata's positivity theorem (see \cite[Theorem 2.4]{fujino-application}) plays a key role as the main ingredient of the proof. 
In our complex analytic setting, one can replace it with \cite[Theorem 21.4]{fujino-bchm} to complete the proof.
\end{proof}

\begin{cor}\label{f-cor7.6}
Let $\varphi\colon X\to Z$ be a projective morphism of complex analytic spaces with $\varphi_*\mathcal O_X\simeq \mathcal O_Z$ such that $(X, \Delta)$ is kawamata log terminal, $-(K_X+\Delta)$ is $\varphi$-ample, and $K_Z$ is $\mathbb Q$-Cartier. 
Then $Z$ is kawamata log terminal. 
\end{cor}

\begin{proof}
Since the problem is local on $Z$, throughout this proof we will freely shrink $Z$ without explicit mention. 
By perturbing $\Delta$ slightly, we may assume that $K_X+\Delta$ is $\mathbb Q$-Cartier. 
Then we have an effective $\mathbb Q$-divisor $D$ on $X$ such that $D\sim _{\mathbb Q} -(K_X+\Delta)$ and that $(X, \Delta+D)$ is still kawamata log terminal. 
Then, by Theorem \ref{f-thm7.5}, we obtain that $Z$ is kawamata log terminal. 
\end{proof}

\section{Supplementary comments}\label{x-sec8}

In this final section, we provide some supplementary comments 
on \cite{fujino-abundance} and 
\cite{fujino-gongyo} for the reader's convenience. 

\begin{say}\label{x-say8.1} 
In \cite[2.20]{fujino-gongyo} and the proof of 
\cite[Theorem 4.3]{fujino-gongyo}, we claim that the 
results in \cite[Section 2]{fujino-abundance} can be 
freely used based on \cite{bchm}. However, in order to 
prove \cite[Proposition 2.1]{fujino-abundance} in 
dimension $n \geq 4$ (see also \cite[Remark 2.2]{fujino-abundance}), 
the minimal model program with scaling established in \cite{bchm} is 
not sufficient. The following result is required.

\begin{thm}[{cf.~\cite[Theorem 5.2]{birkar}}]\label{x-thm8.2}
Let $\pi\colon X\to Y$ be a projective surjective morphism 
of normal quasi-projective varieties and let $(X, \Delta)$ be 
a $\mathbb Q$-factorial divisorial log terminal pair such that 
$K_X+\Delta\sim _{\mathbb Q, \pi}0$. 
Assume that $\pi(\lfloor \Delta\rfloor)\subsetneq Y$, 
that is, $\lfloor \Delta\rfloor$ is vertical with respect to 
$\pi$. 
Then $(X, \Delta-\varepsilon \lfloor \Delta\rfloor)$ has a 
good minimal model over $Y$ for 
every rational number $\varepsilon$ 
with $0<\varepsilon \leq 1$. 
In particular, every $(K_X+\Delta-\varepsilon \lfloor 
\Delta\rfloor)$-minimal model program with ample 
scaling over $Y$ always terminates. 
\end{thm}

If $\pi(\lfloor \Delta\rfloor)=Y$, 
then
$K_X + \Delta - \varepsilon \lfloor \Delta\rfloor$ 
is not $\pi$-pseudo-effective
for any rational number $\varepsilon$ with $0 < \varepsilon \leq 1$.
In this case, the minimal model program established in \cite{bchm}
is sufficient to prove \cite[Proposition 2.1]{fujino-abundance} in 
dimension $n \geq 4$.
Theorem \ref{x-thm8.2} follows from the results in \cite{birkar}.
We note that the pair $(X, \Delta - \varepsilon \lfloor \Delta\rfloor)$ is 
kawamata log terminal
for every rational number $\varepsilon$ 
with $0 < \varepsilon \leq 1$.
Hence, to prove Theorem \ref{x-thm8.2},
we do not require any deep results 
related to the abundance conjecture for log canonical pairs. 
There is no circular reasoning even if one uses \cite[Theorem 5.2]{birkar}
in the context of \cite{fujino-gongyo}.
For further details, see \cite[Theorem 5.2]{birkar}.
We also point out that the most general 
result in this direction is treated in \cite{hashizume1}. 
By combining \cite{bchm} and Theorem \ref{x-thm8.2},
we are free to apply the results 
in \cite[Section 2]{fujino-abundance} for dimension $n \geq 4$.
Therefore, there are no significant issues in the 
arguments of \cite{fujino-gongyo}. 
In the present paper, in 
Step \ref{p-prop4.6-step1} of the proof of Proposition \ref{p-prop4.6},
we use \cite[Theorem 1.2]{enokizono-hashizume-mmp} instead of 
Theorem \ref{x-thm8.2}.
The minimal model program developed in \cite{fujino-bchm} 
is insufficient for the proof of Proposition \ref{p-prop4.6}.
\end{say}

\begin{say}\label{x-say8.3}
We make a small remark 
on \cite[Lemma 2.3]{fujino-abundance} 
for the reader's convenience. 
In the proof of \cite[Lemma 2.3]{fujino-abundance}, 
we claim that there exists a $\mathbb Q$-divisor 
$P$ on $V$ satisfying $K_{D_i}+\mathrm{Diff}(\Delta-D_i)
=u|_{D_i}^*(K_V+P)$. However, it is not clear 
when $D_1$ is irreducible and the mapping degree $\deg [D_1:V]=2$. 
In that case, we can not apply 
\cite[12.3.4 Theorem]{flips-and-abundance}. 

\begin{ex}\label{x-ex8.4}
We put $Z:=\mathbb P^1\times \mathbb P^1$. 
Let $\Delta$ be a general 
member of 
$|p^*_1\mathcal O_{\mathbb P^1}(2)
\otimes p^*_2\mathcal O_{\mathbb P^1}(2)|$, 
where $p_i$ is the $i$th projection for $i=1, 2$. 
Then $\Delta$ is a smooth elliptic curve and $K_Z+\Delta\sim 0$. 
We consider the first projection 
$h\colon Z\to R:=\mathbb P^1$. 
In this setting, 
$u:=h\colon Z\to V:=R$ is 
a $(K_Z+\Delta-\varepsilon \lfloor 
\Delta\rfloor)$-negative extremal 
Fano contraction over $R$. 
Of course, the horizontal part $\Delta^h=:D_1$ of $\lfloor 
\Delta\rfloor$ is irreducible and the mapping degree $\deg[D_1: V]$ is two. 
In \cite[Lemma 2.3]{fujino-abundance}, 
we claim that there exists an effective $\mathbb Q$-divisor 
$P$ on $V$ such that $K_{D_1}=u|^*_{D_1}(K_V+P)$ holds 
without explaining it explicitly. 
It is somewhat misleading when $D_1$ is irreducible 
with $\deg[D_1:V]=2$. 
\end{ex}

Fortunately, as shown in the proof of 
Lemma \ref{p-lem4.5} in the present paper, 
it is not necessary to construct a $\mathbb Q$-divisor 
$P$ on $V$ in Case (III). Therefore, there are no 
significant difficulties in the proof 
of \cite[Lemma 2.3]{fujino-abundance}.
\end{say}


\end{document}